\documentclass[bj,preprint]{imsart}

\usepackage{mathrsfs}
\usepackage{amsfonts}
\usepackage{amsmath}
\usepackage{amsthm}
\usepackage{amssymb}
\usepackage{color}
\usepackage{rotating}
\usepackage{epsfig, subfig, graphicx,}
\usepackage{amsthm,amsmath}
\usepackage[authoryear,round,compress,sort]{natbib}
\usepackage{tabularx, booktabs}
\usepackage{multirow}
\RequirePackage{bm}
\usepackage{mathrsfs}
\usepackage{dsfont}

\startlocaldefs
\endlocaldefs

\def\tr{\mathop{\text{tr}}\kern.2ex}

\def\P{{\mathbb P}}
\def\E{{\mathbb E}}

\def\Z{{\mathbb Z}}

\def\supp{\mathop{\text{supp}}}

\long\def\comment#1{}

\def\tr{\mathop{\text{Tr}}}

\def\ind{\mathds{1}}

\let\hat\widehat
\let\tilde\widetilde

\newcommand{\bs}{\bm{s}}

\newcommand{\bu}{\bm{u}}
\newcommand{\bv}{\bm{v}}
\newcommand{\bw}{\bm{w}}

\newcommand{\Db}{\mathbf{D}}

\newcommand{\Ib}{\mathbf{I}}

\newcommand{\Mb}{\mathbf{M}}

\newcommand{\Ob}{\mathbf{O}}

\newcommand{\bG}{\bm{G}}

\newcommand{\bR}{\bm{R}}

\newcommand{\bU}{\bm{U}}

\newcommand{\bX}{\bm{X}}
\newcommand{\bY}{\bm{Y}}
\newcommand{\bZ}{\bm{Z}}

\newcommand{\cB}{\mathcal{B}}

\newcommand{\cF}{\mathcal{F}}
\newcommand{\cG}{\mathcal{G}}

\newcommand{\cM}{\mathcal{M}}

\newcommand{\cQ}{\mathcal{Q}}
\newcommand{\cR}{\mathcal{R}}
\newcommand{\cS}{{\mathcal{S}}}
\newcommand{\cT}{{\mathcal{T}}}

\newcommand{\bSigma}{\bm{\Sigma}}

\newcommand*{\zero}{{\bm 0}}

\numberwithin{equation}{section}
\theoremstyle{plain}

\newtheorem{lemma}{Lemma}[section]

\newtheorem{theorem}{Theorem}[section]
\newtheorem{assumption}{Assumption}[section]

\newtheorem{remark}{Remark}[section]

\newcommand{\norm}[1]{\|#1\|}

\providecommand{\norm}[1]{\vvvert#1\vvvert}

\newcommand{\bel}{\begin{eqnarray}\label}
\newcommand{\eel}{\end{eqnarray}}
\newcommand{\bes}{\begin{eqnarray*}}
\newcommand{\ees}{\end{eqnarray*}}

\def\reals{{\mathbb{R}}}

\def\T{{\sf T}}

\let\bar\overline

\usepackage[usenames,dvipsnames,svgnames,table]{xcolor}
\usepackage[colorlinks=true,
            linkcolor=blue,
            urlcolor=blue,
            citecolor=blue]{hyperref}
            
\begin{document}

\begin{frontmatter}

\title{On Gaussian comparison inequality and its application to spectral analysis of large random matrices}

\runtitle{Gaussian Comparison Inequality and Random Matrix Theory}

\begin{aug}

\author{\fnms{Fang} \snm{Han}\thanksref{t1}\ead[label=e1]{fanghan@uw.edu}},
\author{\fnms{Sheng} \snm{Xu}\thanksref{t3}\ead[label=e2]{sheng.xu@yale.edu}},
\and
\author{\fnms{Wen-Xin} \snm{Zhou}\thanksref{t4}\ead[label=e3]{stevewenxin@hotmail.com}}

\address[t1]{Department of Statistics, University of Washington, Seattle, WA 98195, USA. \printead{e1}}
\address[t3]{Department of Statistics, Yale University, New Haven, CT 06511, USA. \printead{e2}}
\address[t4]{Department of Mathematics, University of California, San Diego, La Jolla, CA 92093, USA. \printead{e3}}

\runauthor{F. Han, S. Xu, and W.-X. Zhou}

\affiliation{University of Washington, Johns Hopkins University, and Princeton University}

\end{aug}

\begin{abstract}
Recently, Chernozhukov, Chetverikov, and Kato [{\em Ann. Statist.} {\bf 42} (2014) 1564--1597] developed a new Gaussian comparison inequality for approximating the suprema of empirical processes. This paper exploits this technique to devise sharp inference on spectra of large random matrices. In particular, we show that two long-standing problems in random matrix theory can be solved: (i) simple bootstrap inference on sample eigenvalues when true eigenvalues are tied; (ii) conducting two-sample Roy's covariance test in high dimensions. To establish the asymptotic results, a generalized $\epsilon$-net argument regarding the matrix rescaled spectral norm and several new empirical process bounds are developed and of independent interest.
\end{abstract}



\end{frontmatter}

\section{Introduction}\label{sec:introduction}

Spectral analysis of large random matrices plays {an important} role in multivariate statistical estimation and testing problems. For example, variances of the principal components are functions of covariance eigenvalues \citep{muirhead2009aspects}, and Roy's largest root test statistic is the spectral distance between the sample covariance and its population counterpart \citep{roy1957some}. 

Asymptotic behaviors of sample covariance eigenvalues have been extensively studied in the literature. When the dimension $d$ is small and the population eigenvalues are distinct, \cite{anderson1963asymptotic} and \cite{waternaux1976asymptotic} proved the asymptotic normality for sample eigenvalues. \cite{fujikoshi1980asymptotic} established the Edgeworth expansion and showed that the convergence rate is of order $O(n^{-1/2} )$ under various settings when $d$ is fixed. For non-Gaussian data,  \cite{waternaux1976asymptotic} and \cite{fujikoshi1980asymptotic} illustrated the effects of skewness and kurtosis on the limiting distribution.

When $d$ is large, \cite{johnstone2001distribution} revealed for Gaussian data that the largest sample eigenvalue, after proper standardization, follows the Tracy-Widom law asymptotically \citep{tracy1996orthogonal}. \cite{johnstone2008multivariate} further proved that the convergence rate to the Tracy-Widom law is of order $O(d^{-2/3})$, which is astonishingly fast. Despite these elegant properties, existing results rely heavily on some simple Gaussian or sub-Gaussian assumptions \citep{peche2009universality,pillai2012edge,bao2015universality}. Their applications to hypothesis testing and constructing confidence intervals under more general settings are largely unknown.

Motivated by the covariance testing problem, the major focus of this paper is to study asymptotic behaviors of a particular type of spectral statistics related to the covariance matrix. Here we are interested in the non-Gaussian setting with the dimension $d$ allowed to grow with the sample size $n$. Specifically, let $\bX_1,\ldots,\bX_n$ be $n$ independent realizations of a $d$-dimensional random vector $\bX$ with mean $\zero$ and covariance matrix $\bSigma\in\reals^{d\times d}$. Denote the sample covariance matrix by $\hat\bSigma = n^{-1} \sum_{i=1}^n \bX_i\bX_i^\T $. We shall derive the limiting distribution and establish bootstrap confidence intervals for the following statistic
\begin{align}\label{eq:statistics}
\hat{Q}_{\max} :=\sup_{\norm{\bv}_2\leq 1, \,\norm{\bv}_0\leq s} \bigg|\frac{\sqrt{n} \, \bv^\T(\hat\bSigma-\bSigma)\bv}{\bv^\T\bSigma\bv} \bigg|,
\end{align}
where $1\leq s\leq d$ is a prespecified integer-valued parameter representing the ``degree of sparsity". The statistic $\hat{Q}_{\max}$ is of general and strong practical interest. By setting $s=d$, it reduces  to the conventional Roy's test statistic $\sqrt{n}\norm{\bSigma^{-1/2}\hat\bSigma\bSigma^{-1/2}-\Ib_d}_2$, where $\norm{\Mb}_2$ denotes the spectral norm of $\Mb$. If $s\leq d-1$, we obtain a generalized version of Roy's test statistic, allowing us to deal with large covariance matrices\footnote{The techniques built in this paper can also be exploited to study the non-normalized version of $\hat{Q}_{\max}$, i.e., $\sup_{\norm{\bv}_2\leq 1,\norm{\bv}_0\leq s}|\bv^\T(\hat\bSigma-\bSigma)\bv|$. We defer to Section \ref{sec:main} for more details.}. 

To study the limiting behavior of $\hat{Q}_{\max}$ in high dimensions, a major insight is to build the connection between the analysis of the maximum eigenvalue and recent developments in extreme value theory. In particular, by viewing the maximum eigenvalue as the extreme value of a specific infinite-state stochastic process, the Gaussian comparison inequality recently developed in \cite{chernozhukov2014gaussian} can be used. New empirical process bounds are established to ensure the validity of the inference procedure. In the end, bootstrap inference follows.

Two interesting observations are discovered. First, in the low-dimensional regime ($d/n \to 0$), the results in this paper solve a long standing question on bootstrap inference of eigenvalues when multiple roots exist \citep{beran1985bootstrap,eaton1991wielandt}. The $m$-out-of-$n$ bootstrap \citep{hall1993inconsistency} is known to be rather sensitive to the choice of $m$. In comparison, the multiplier-bootstrap-based inference procedure used in this paper does not involve any tuning parameter, and is fairly accurate in approximating the distribution of the test statistic. Secondly, it is well-known that Roy's largest root test is optimal against rank-one alternatives \citep{kritchman2009non}. Previously it was unclear whether such a result could be extended to high dimensional settings. This paper demonstrates that such a generalization can be made.

\subsection{Notation}\label{sec:notation}

Throughout the paper, let $\reals$ and $\Z$ denote the sets of real numbers and integers. Let $\mathds{1}(\cdot)$ be the indicator function. Let $\bv=(v_1,\ldots,v_d)^{\T}$ and $\Mb= ( \Mb_{jk} ) \in\reals^{d\times d}$ be a $d$ dimensional real vector and a $d\times d$ real matrix. For sets $I,J \subset \{1,\ldots,d\}$, let $\bv_I$ be the subvector of $\bv$ with entries indexed by $I$, and $\Mb_{I,J}$ be the submatrix of $\Mb$ with entries indexed by $I$ and $J$. We define the vector $\ell_0$ and $\ell_2$ (pseudo-)norms of $\bv$ to be $\norm{\bv}_0 =\sum_j \mathds{1}(v_j\ne 0)$ and $\norm{\bv}_2=\big( \sum_{j=1}^d |v_j |^2 \big)^{1/2}$. We define the matrix spectral ($\ell_2$) norm as $\norm{\Mb}_2= \max_{\bv} \norm{\Mb\bv}_2/\norm{\bv}_2$. For every real symmetric matrix $\Mb$, we define $\lambda_{\max}(\Mb)$ and $\lambda_{\min}(\Mb)$ to be its largest and smallest eigenvalues. For any integer $1\leq s\leq d$ and real symmetric matrix $\Mb$, we define the $s$-sparse smallest and largest eigenvalues of $\Mb$ to be
\[
\lambda_{\min, s}(\Mb )=\inf_{\bv\in\mathbb{V}(s,d)}\bv^{\T}\Mb\bv~~{\rm and}~~\lambda_{\max,s}(\Mb )=\sup_{\bv\in\mathbb{V}(s,d)}\bv^{\T}\Mb\bv,
\] 
where
\begin{align} \label{cdt0}
\mathbb{V}(s,d):= \big\{\bv\in\reals^d: \norm{\bv}_2=1,\norm{\bv}_0\leq s \big\}
\end{align}
is the set of all $s$-sparse vectors on the $d$-dimensional sphere $\mathbb{S}^{d-1}$.
Moreover, we write $\gamma_s(\Mb)=\sqrt{\lambda_{\max,s}(\Mb )/\lambda_{\min,s }(\Mb )}$ for any positive definite matrix $\Mb$. For any $\bv\in\reals^d$ and positive definite real-valued matrix $\Mb$, we write
\[
\|\bv\|_{\Mb}=(\bv^{\T}\Mb\bv)^{1/2}~~{\rm and}~~\bv_{\Mb}=\bv/\|\bv\|_{\Mb}.
\]

For any random vectors $\bX,\bY\in \reals^d$, we write $\bX\stackrel{{\sf d}}{=}\bY$ if $\bX$ and $\bY$ are identically distributed. We use $c, C$ to denote absolute positive constants, which may take different values at each occurrence. For any two real sequences $\{a_n\}$ and $\{b_n\}$, we write $a_n \lesssim b_n$, $a_n=O(b_n)$, or equivalently $b_n \gtrsim a_n$, if there exists an absolute constant $C$ such that $|a_n|\leq C|b_n|$ for any large enough $n$. We write $a_n \asymp b_n$ if both $a_n \lesssim b_n$ and $a_n\gtrsim b_n$ hold. We write $a_n=o(b_n)$ if for any absolute constant $C$, we have $|a_n|\leq C|b_n|$ for any large enough $n$. We write $a_n=O_{\mathbb{P}}(b_n)$ and $a_n=o_{\mathbb{P}}(b_n)$ if $a_n=O(b_n)$ and $a_n=o(b_n)$ hold stochastically. For arbitrary positive integer $n$, we write $[n]=\{a\in\Z: 1\leq a\leq n\}$. For any set $\mathbb{A}$, denote by $|\mathbb{A}|$ its cardinality and $\supp(\mathbb{A})$ its support. For any $a,b\in\reals$, we write $a\vee b=\max(a, b)$.

\subsection{Structure of the paper}

The rest of this paper is organized as follows. Section \ref{sec:main} provides the main results and technical tools involved. Sections \ref{sec:app1} and \ref{sec:app2} give two applications of the main results. In particular, Section \ref{sec:app1} discusses the application to bootstrap inference of largest and smallest eigenvalues for spherical distributions. Section \ref{sec:app2} extends the main results to conduct the two-sample Roy's largest root test. {In Section \ref{sec:discussion}, we conclude the paper with a short discussion}. {The technical proofs are relegated to Section \ref{sec:proof}}.

\section{Main results}\label{sec:main}

Let $\bX_1,\ldots,\bX_n$ be independent and identically distributed (i.i.d.) realizations of $\bX\in\reals^d$ with mean $\zero$ and covariance matrix $\bSigma$, and let $\hat\bSigma$ be the sample covariance matrix. Define $\hat Q_{\bv}$ and $\hat Q_{\max}$ to be the normalized rank-one projection and normalized $s$-sparse largest singular value of $\hat\bSigma-\bSigma$, given respectively by
\[
\hat Q_{ \bv} = \frac{\sqrt{n} \, \bv^\T(\hat\bSigma-\bSigma)\bv}{\bv^\T\bSigma\bv}~~{\rm and}~~\hat Q_{\max}  =\sup_{\bv\in\mathbb{V}(s,d)}|\hat Q_{ \bv}|.
\]
We aim to derive the limiting distribution of $\hat Q_{\max}$. Of note, when setting $s=d$ and assuming the positive definiteness of $\bSigma$, we have
\[
\hat Q_{\max} =\sqrt{n} \, \norm{\bSigma^{-1/2}\hat\bSigma\bSigma^{-1/2}-\Ib_d}_2,
\]
which coincides with Roy's largest root test statistic \citep{roy1957some,johnstone2013roy}. The statistic $\hat Q_{\max}$ is of strong practical interest. We shall discuss in Sections~\ref{sec:app1} and \ref{sec:app2} two applications based on its limiting properties stated below.

To derive the limiting distribution of $\hat Q_{\max}$, we impose the following two assumptions.
\begin{assumption}\label{cdt1}
There exists a random vector $\bU\in\reals^d$ satisfying $\E ( \bU )=\zero$ and $\E( \bU\bU^{\T} )=\Ib_d$, such that
\[
\bX=\bSigma^{1/2 }\bU~~{\rm and}~~K_1:=\sup_{\bv\in\mathbb{S}^{d-1}}\|\bv^{\T}\bU\|_{\psi_2}<\infty.
\]
Here $\norm{\cdot}_{\psi_2}$ stands for the standard Orlicz norm with respect to the function $\psi_2(x) :=\exp(x^2)-1$ \citep{van1996weak}.
\end{assumption}
\begin{assumption}\label{cdt2}
$\{\bX_i\}_{i=1}^n$ are independent realizations of $\bX$.
\end{assumption}

Assumptions~\ref{cdt1} and \ref{cdt2} are mild and are regularly imposed in the literature. Note that the sub-Gaussian condition in Assumption \ref{cdt1} can be easily relaxed at the cost of a more stringent scaling constraint on $(n,d,s)$ \citep{cai2013two}. Assumption \ref{cdt2} can also be slightly relaxed. Such relaxations are beyond the scope of this paper, and we will not pursue the details here.

With Assumptions~\ref{cdt1} and \ref{cdt2} satisfied, the following theorem gives a Gaussian comparison result regarding the limiting distribution of $\hat{Q}_{\max}$. Below, for an arbitrary set $\mathbb{A}$ equipped with a metric $\rho$, we call $\mathbb{N}_{\epsilon}$ an $\epsilon$-net of $\mathbb{A}$ if for every $a\in \mathbb{A}$ there exists some $a' \in \mathbb{N}_{\epsilon}$ such that $\rho(a , a' )\leq \epsilon $. For each $s \in [d]$, with slight abuse of notation, we write $\gamma_s  = \gamma_s (\bSigma)$ for simplicity. By Lemma~\ref{lem:discret} in Section~\ref{sec:proof}, for any $\epsilon\in (0,1)$, there exists an $\epsilon$-net $\mathbb{N}^0_{\epsilon }=\{\bv_j:j=1,\ldots,p_{\epsilon}^0 \}$ of $\mathbb{V}(s,d)$ equipped with the Euclidean metric, with its cardinality $p_\epsilon^0 =|\mathbb{N}^0_{\epsilon }|$ satisfying $\log p^0_\epsilon \lesssim s\log(ed/\epsilon  s)$.

\begin{theorem}\label{thm:limiting_dist}
Let Assumptions \ref{cdt1} and \ref{cdt2} be satisfied and put $\epsilon_1 =(n\gamma_s)^{-1}$. Then, for any $\epsilon_1$-net $\mathbb{N}_{\epsilon_1}$ of $\mathbb{V}(s,d)$ with cardinality $p_{\epsilon_1} =|\mathbb{N}_{\epsilon_1}|$, there exists a $p_{\epsilon_1}$-dimensional centered Gaussian random vector $(G_1,\ldots,G_{p_{\epsilon_1}})^{\T}$ satisfying $\E ( G_j G_k )=\E( R_{j} R_{k} )$ for $1\leq j, k \leq p_{\epsilon_1}$ with $R_{j} = \bv_j^{\T}(\bX\bX^{\T}-\bSigma)\bv_j /\bv_j^{\T}\bSigma\bv_j$, such that 
\begin{align}\label{eq:21}
\sup_{t\geq0}\bigg|\P( \hat{Q}_{\max}  \leq t)-\P\bigg(\max_{1\leq j\leq p_{\epsilon_1} }|G_{j}|\leq t\bigg)\bigg|\leq CK_2^2  \frac{\{   \gamma_n(s,d) \vee \log  p_{\epsilon_1} \} ^{9/8}}{n^{1/8}} ,
\end{align}
where $C>0$ is an absolute constant, $K_2 := K_1^2+1$, and $\gamma_n(s,d) := s\log (\gamma_s \cdot ed/s) \vee s\log n$.
\end{theorem}

There are several interesting observations drawn from Theorem \ref{thm:limiting_dist}. First, as long as $  s\log (\gamma_s \cdot ed/s) \vee s\log n \vee \log  p_\epsilon=o(n^{1/9})$ for a properly chosen $\epsilon$, the distribution of $\hat{Q}_{\max}$ can be well approximated by that of the maximum of a Gaussian sequence. It is worth noting that no parametric assumption is imposed on the data generating scheme. Secondly, the result in Theorem \ref{thm:limiting_dist}, though not reflecting the exact limiting distribution of $\hat{Q}_{\max}$, sheds light on its asymptotic behavior. Following the standard extreme value theory, when $s=1$ and the covariance matrix $\bSigma$ is sparse, $\hat{Q}_{\max}$ follows a Gumbel distribution asymptotically as $d \to \infty$ \citep{cai2013two}. Thirdly, we note that when $\bSigma=\Ib_d$, the techniques used to prove Theorem \ref{thm:limiting_dist} can be adapted to derive the limiting distributions of extreme sample eigenvalues. See Section \ref{sec:app1} for details.

The detailed proof of Theorem \ref{thm:limiting_dist} is involved. Hence, a heuristic sketch is useful. A major ingredient stems from a Gaussian comparison inequality recently developed by \cite{chernozhukov2014gaussian}.

\begin{lemma}[Gaussian comparison inequality]\label{lem:coupling_ineq}
Let $\bX_1,\ldots,\bX_n$ be independent random vectors in $\reals^d$ with mean zero and finite absolute third moments, that is, $\E ( X_{ij} )=0$ and $\E ( |X_{ij}|^3 ) <\infty$ for all $1\leq i\leq n$ and $1\leq j\leq d$. Consider the statistic $Z:=\max_{1\leq j\leq d}\sum_{i=1}^n X_{ij}$. Let $\bY_1,\ldots,\bY_n$ be independent random vectors in $\reals^d$ with $\bY_i\sim N_d(\zero,\E ( \bX_i\bX_i^{\T} ) )$, $1\leq i\leq n$. Then, for every $\delta>0$, there exists a random variable $\tilde{Z}\stackrel{{\sf d}}{=}\max_{1\leq j\leq d}\sum_{i=1}^n Y_{ij}$ such that
\begin{align*}
\P(|Z-\tilde{Z}|\geq16\delta)\lesssim  \delta^{-2} \log(d n)  \{  D_1+ \delta^{-1} \log(d n) (D_2+D_3) \} + n^{-1} \log n ,
\end{align*}
where we write
\begin{align*}
D_1 =\E\bigg[\max_{1\leq j,l\leq d}\bigg|\sum_{i=1}^n \{ X_{ij}X_{il}-\E ( X_{ij}X_{il}) \} \bigg|\bigg],~~D_2=\E\bigg( \max_{1\leq j\leq d}\sum_{i=1}^n|X_{ij}|^3 \bigg), \\
D_3 =\sum_{i=1}^n \E\bigg[\max_{1\leq j\leq d}|X_{ij}|^3\cdot \mathds{1}\bigg\{ \max_{1\leq j\leq d}|X_{ij}|>\frac{\delta}{\log(d  n)}\bigg\} \bigg].
\end{align*}
\end{lemma}

In view of Lemma \ref{lem:coupling_ineq} and the fact that $\hat{Q}_{\max}$ is the supremum of an infinite-state process, the proof can be divided into three steps. In the first step, we prove that the difference between $\hat{Q}_{\max}$ and its ``discretized version" is negligible asymptotically. This is implied by the following generalized $\epsilon$-net argument for the rescaled spectral norm. It extends the standard $\epsilon$-net argument \citep{vershynin2010introduction}.
\begin{lemma}\label{lem:obs}
For any $\bv,\tilde\bv\in\mathbb{V}(s,d)$ with the same support, positive definite matrix $\bSigma$, and any real symmetric matrix $\Mb$, we have
\begin{align*}
\big||\bv_{\bSigma}^{\T}\Mb\bv_{\bSigma}|-|\tilde\bv_{\bSigma}^{\T}\Mb\tilde\bv_{\bSigma}|\big|\leq 2\gamma_s \|\bv-\tilde\bv\|_2 \sup_{\bv\in\mathbb{V}(s,d)}|\bv_{\bSigma}^{\T}\Mb\bv_{\bSigma}|.
\end{align*}
\end{lemma}

In the second step, we show that this discretized version of $\hat{Q}_{\max}$ converges in distribution to the maximum of a finite Gaussian sequence. This can be achieved by exploiting Lemma \ref{lem:coupling_ineq}.  Lastly, anti-concentration bounds  \citep{chernozhukov2014comparison} are established to bridge the gap between the distributions of $\hat{Q}_{\max}$ and its discretized version. 
The  complete proof is provided in Section \ref{sec:proof}.

The asymptotic result in Theorem \ref{thm:limiting_dist} is difficult to use in practice. To estimate the limiting distribution of $\hat{Q}_{\max}$ empirically, bootstrap approximation is preferred. For any $\bv\in\mathbb{V}(s,d)$, define
\begin{align}\label{eq:MBdef}
\hat{B}_{\bv} =\frac{1}{\sqrt{n}} \frac{  \sum_{i=1}^n\xi_i(\bv^{\T} \bX_i\bX_i^{\T}\bv -\bv^{\T}\bSigma\bv )}{\bv^{\T}\bSigma\bv} ~~~{\rm and}~~~ \hat B_{\max} = \sup_{\bv\in\mathbb{V}(s,d)}| \hat{B}_{\bv} |,
\end{align}
where $\xi_1,\ldots,\xi_n$ are i.i.d. standard normal that are independent of $\{\bX_i\}_{i=1}^n$. We use the conditional distribution of $\hat B_{\max}$ given the data to approximate the distribution of $\hat{Q}_{\max}$. The next theorem characterizes the validity of bootstrap approximation.

\begin{theorem}\label{thm:multiplier_bootstrap}
Let Assumptions \ref{cdt1} and \ref{cdt2} be satisfied, and assume that $ \gamma_n(s,d)=s\log (\gamma_s \cdot ed/s) \vee s\log n=o(n^{1/9})$ as $n\to \infty$. Then, there exists a sufficiently large absolute constant $C>0$ such that
\begin{align*}
\P\bigg\{ \sup_{t\geq0}\Big|\P\big(\hat{Q}_{\max}\leq t\big)-\P\big(\hat B_{\max}\leq t  \, | \bX_1,\ldots,\bX_n\big)\Big|\geq C  \frac{ \gamma^{9/8}_n(s,d)}{n^{1/8}} \bigg\} =o(1).
\end{align*}
In other words, we have
\begin{align*}
 \sup_{t\geq0}\big|\P (\hat{Q}_{\max}\leq t )-\P ( \hat B_{\max} \leq t \,  |  \bX_1,\ldots,\bX_n ) \big|  = o_{\mathbb{P}}(1) .
\end{align*}
\end{theorem}

The proof of Theorem \ref{thm:multiplier_bootstrap} heavily relies on characterizing the convergence rates of sub-Gaussian fourth-order terms. We defer this result and the detailed proof of Theorem \ref{thm:multiplier_bootstrap} to Section \ref{sec:proof}.

\medskip

The rest of this section gives asymptotic results for the non-normalized version of $\hat{Q}_{\max}$. To this end, let $\tilde Q_{\bv}$ and $\tilde Q_{\max}$ be the rank-one projection and $s$-sparse largest singular value of $\hat\bSigma-\bSigma$, given respectively by
\[
\tilde Q_{\bv} = \sqrt{n}\, \bv^\T(\hat\bSigma-\bSigma)\bv~~{\rm and}~~\tilde Q_{\max}  =\sup_{\bv\in\mathbb{V}(s,d)}| \tilde Q_{ \bv}|.
\]
Technically speaking, $\tilde Q_{\max}$ is a simpler version of $\hat{Q}_{\max}$. We show that, under an additional eigenvalue assumption, $\tilde Q_{\max}$ converges weakly to the extreme of a Gaussian sequence. In particular, the following condition assumes that the $s$-sparse (restricted) largest eigenvalue of $\bSigma$ is upper bounded by an absolute constant.

\begin{assumption}\label{cdt4}
There exists an absolute constant $L>0$ such that $\lambda_{\max, s}(\bSigma)\leq L$.
\end{assumption}

We define, for any $\bv \in \mathbb{V}(s,d)$,
\begin{align}\label{eq:MBdef2}
\tilde B_{ \bv} =\frac{1}{\sqrt{n}} \sum_{i=1}^n\xi_i(\bv^{\T} \bX_i\bX_i^{\T} \bv - \bv^{\T} \bSigma  \bv )  ~~~{\rm and}~~~\tilde B_{\max} =\sup_{\bv\in\mathbb{V}(s,d)}|\tilde B_{ \bv}|,
\end{align}
where $\xi_1,\ldots,\xi_n$ are i.i.d. standard normal random variables independent of $\{\bX_i\}_{i=1}^n$. The following theorem gives the Gaussian approximation result for $\tilde Q_{\max}$.

\begin{theorem}\label{thm:limiting_dist2}
Let Assumptions \ref{cdt1}--\ref{cdt4} be satisfied and set $\epsilon_2 = n^{-1}$. Then, for any $\epsilon_2$-net $\mathbb{N}_{\epsilon_2}$ of $\mathbb{V}(s,d)$ with cardinality $p_{\epsilon_2} =|\mathbb{N}_{\epsilon_2}|$, there exists a $p_{\epsilon_2}$-dimensional centered Gaussian random vector $(\tilde G_1,\ldots, \tilde G_{p_{\epsilon_2}})^{\T}$ satisfying $\E( \tilde G_j \tilde G_k ) =\E ( \tilde R_{j} \tilde R_k )$ for $1\leq j, k \leq p_{\epsilon_2}$ with $\tilde R_{j} :=\bw_j^{\T}(\bX\bX^{\T}-\bSigma)\bw_j$, such that
\begin{align}\label{eq:24}
\sup_{t\geq0}\bigg| \P( \tilde{Q}_{\max} \leq t )-\P\bigg(\max_{1\leq j\leq p_{\epsilon_2}}| \tilde G_{j} |\leq t\bigg)\bigg|\leq C_LK_2^2   \frac{  \delta^{9/8}_n(s,d)}{n^{1/8}} ,
\end{align}
where $C_L>0$ is a constant depending only on $L$, $K_2 =K_1^2+1$, and $\delta_n(s,d):=s\log (ed/s) \vee s \log n$. In addition, if $(s,d,n)$ satisfies $s\log (ed/s) \vee s \log n = o(n^{1/9})$ as $n\to \infty$, then there exists an absolute constant $C>0$ large enough such that
\begin{align}\label{eq:25}
\P\bigg\{ \sup_{t\geq0}\big|\P(\tilde Q_{\max} \leq t  )-\P ( \tilde B_{\max} \leq t \, | \bX_1,\ldots,\bX_n )\big|\geq C \frac{ \delta^{9/8}_n(s,d)}{n^{1/8}} \bigg\} =o(1).
\end{align}
In other words, we have
\begin{align*}
 \sup_{t\geq0}\big|\P (\tilde{Q}_{\max}\leq t )-\P ( \tilde B_{\max} \leq t \, |  \bX_1,\ldots,\bX_n ) \big|  = o_{\mathbb{P}}(1) .
\end{align*}
\end{theorem}


\begin{remark}
{\rm
By comparing Theorems \ref{thm:limiting_dist} and \ref{thm:limiting_dist2}, we immediately observe some difference between the properties of $\hat{Q}_{\max}$ and $\tilde Q_{\max}$. To ensure the validity of the multiplier bootstrap approximation for $\hat{Q}_{\max}$, we only require $ s\log (\gamma_s   ed/s) \vee s\log n =o(n^{1/9})$, and thus allow $\lambda_{\max, s}(\bSigma)$ to grow quickly.  In contrast, the bootstrap approximation consistency for $\tilde Q_{\max}$ relies on $C_L$, a constant of the same order as $\lambda_{\max, s}^2(\bSigma)$. }
\end{remark}

\section{Application I: bootstrap inference on largest  and smallest eigenvalues for spherical distributions}\label{sec:app1}

A direct application of Theorem \ref{thm:limiting_dist} is on inferring extreme sample eigenvalues of spherical distributions. A random vector is said to be spherically distributed if its covariance matrix is proportional to the identity. Note that this definition is slightly different from its counterpart in robust statistics, where a more stringent rotation-invariant property is required \citep{fang1990symmetric}.

It is known that when multiple roots exist (i.e., the population eigenvalues are not distinct), the sample eigenvalues are not asymptotically normal even under the Gaussian assumption \citep{anderson1963asymptotic}. \cite{waternaux1976asymptotic} and \cite{tyler1983asymptotic} showed that inference is even more challenging for non-Gaussian data as the limiting distributions of the sample eigenvalues rely on the skewness and kurtosis of the underlying distribution. Estimation of these parameters is statistically costly. Bootstrap methods are hence recommended for conducting inference. 

In the presence of multiple roots, Beran and Srivastava \citep{beran1985bootstrap} pointed out that the nonparametric bootstrap for eigenvalue inference is inconsistent. The $m$-out-of-$n$ bootstrap \citep{hall1993inconsistency} and its modification \citep{hall2009tie} are hence proposed to correct this. The implementation, however, is complicated since tuning parameters are involved.  

Based on Theorem \ref{thm:multiplier_bootstrap}, we show that a simple multiplier bootstrap method leads to asymptotically valid inference for extreme eigenvalues, as stated in the next theorem.
\begin{theorem}\label{thm:limiting_dist3}
Suppose that Assumptions \ref{cdt1} and \ref{cdt2} hold. In addition, assume that $\bSigma=\sigma^2\Ib_d$ with $\sigma^2>0$ an absolute constant. Then, as long as $d=o(n^{1/9})$,
\begin{align*}
\sup_{t}\bigg|\P\big\{ \lambda_{\max}(\hat\bSigma ) \!\leq \sigma^2+t\big\}\! -\! \P\bigg[ \lambda_{\max}\bigg\{  \frac{1}{n}\sum_{i=1}^n\xi_i(\bX_i\bX_i^\T\!-\!\sigma^2\Ib_d) \bigg\} \!\leq t \,\bigg| \bX_1,\ldots,\bX_n\bigg] \bigg|\!=\!o_{\mathbb{P}}(1),
\end{align*}
and
\begin{align*}
\sup_{t}\bigg|\P\big\{ \lambda_{\min}(\hat\bSigma) \!\geq \sigma^2+t\big\} \!-\! \P\bigg[ \lambda_{\min}\bigg\{ \frac{1}{n}\sum_{i=1}^n\xi_i(\bX_i\bX_i^\T\!-\!\sigma^2\Ib_d) \bigg\} \!\leq t \, \bigg| \bX_1,\ldots,\bX_n\bigg] \bigg|\!=\!o_{\mathbb{P}}(1).
\end{align*}
Here $\{\xi_i\}_{i=1}^n$ forms an independent standard Gaussian sequence independent of the data. 
\end{theorem}

\begin{figure}[tp!]
\begin{center}
\begin{tabular}{c}
\vspace{-0.05in}
\includegraphics[width=0.95\textwidth,angle=0]{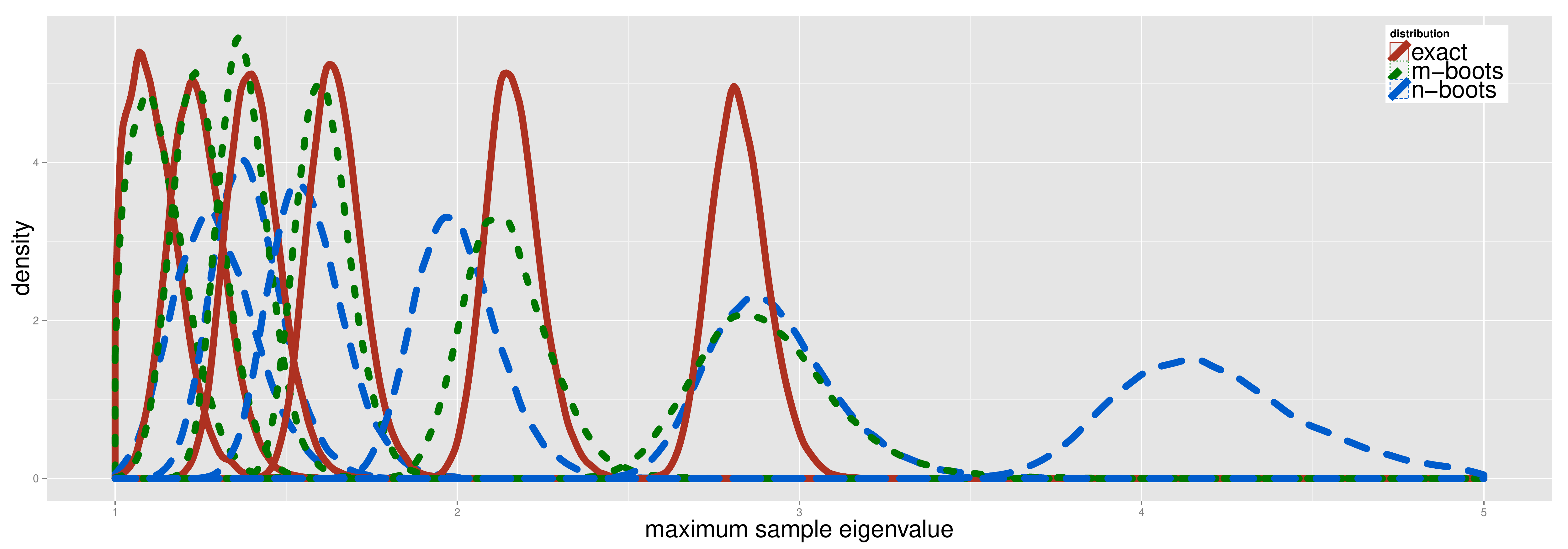} 
\\[-0pt]
$n=200$
\\
\includegraphics[width=0.95\textwidth,angle=0]{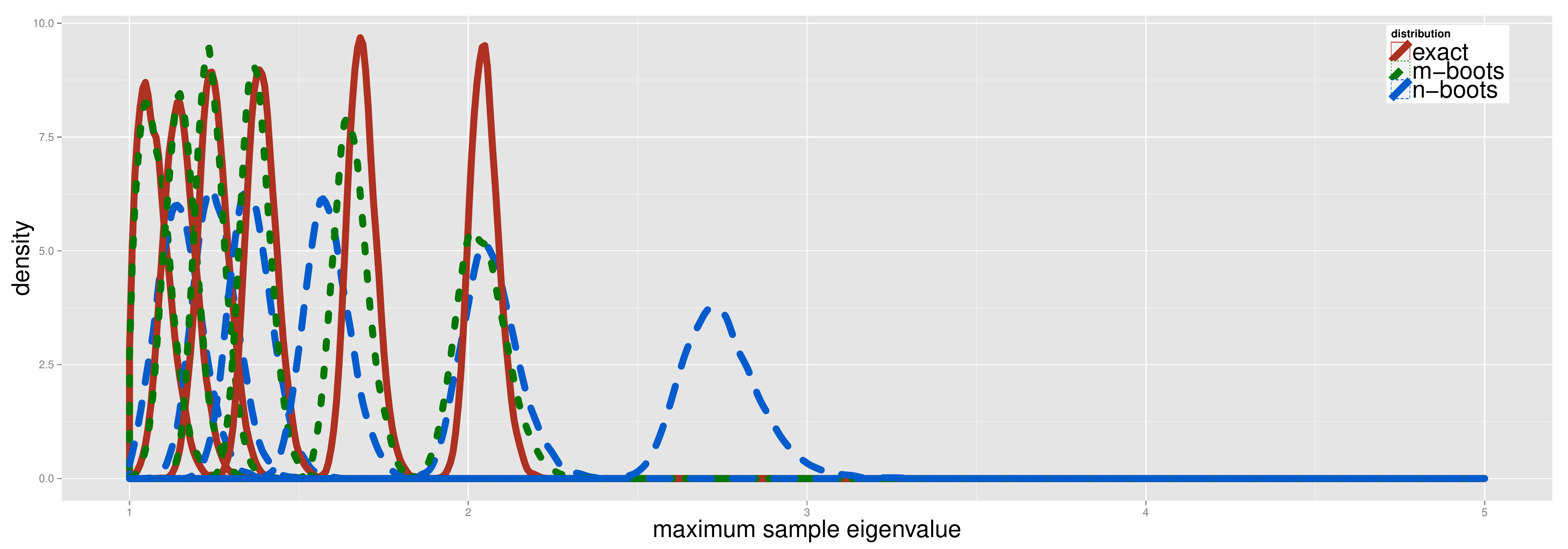} 
\\[-4pt]
$n=500$
\\
\includegraphics[width=0.95\textwidth,angle=0]{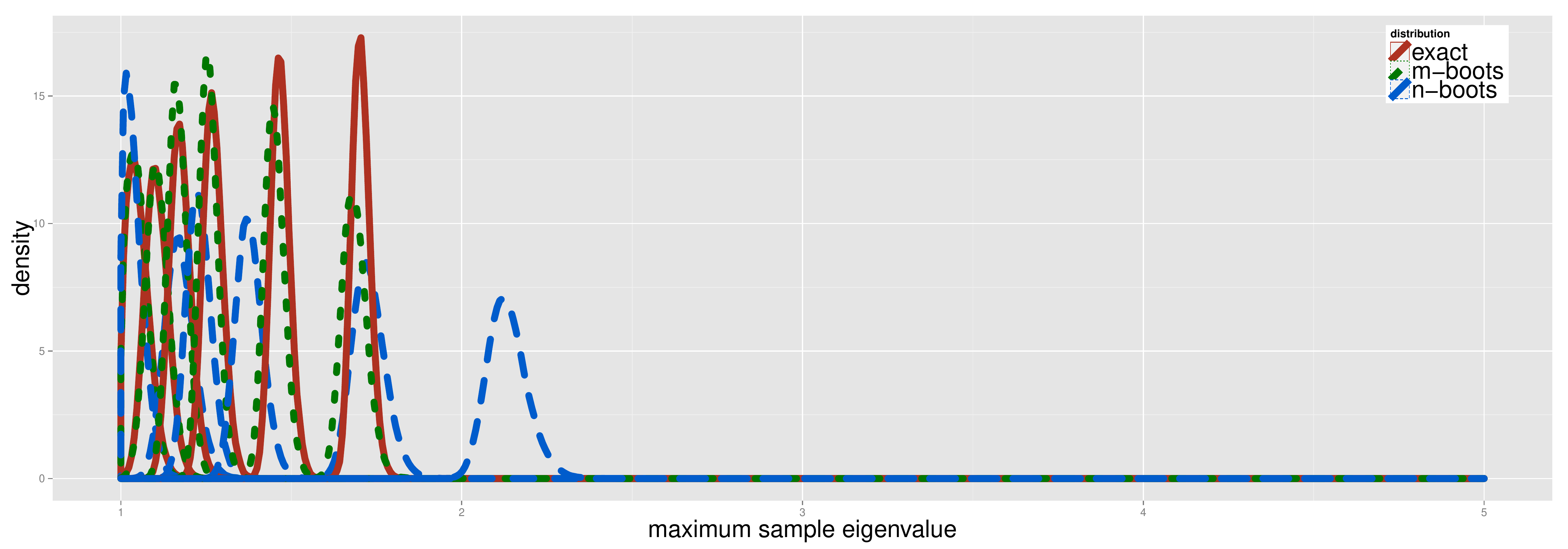} 
\\[-4pt]
$n=1,000$
\\[-10pt]
\end{tabular}
 \end{center}
 \caption{Density plots of three approximation distributions based on $n=200, 500$, and $1,000$ data points (from top to bottom) randomly drawn from the standard Gaussian distribution. These three distributions are: (i) the exact distribution of $\lambda_{\max}(\hat\bSigma)$ (red, solid, denoted as ``exact"); (ii) the distribution from the multiplier-bootstrap method (green, dotted, denoted as ``m-boots"); (iii) the distribution from the nonparametric-bootstrap method (blue, broken, denoted as ``n-boots"). Within each graph, for each distribution, the curves correspond to the setting $d=2,4,10,20,50,100$ from left to right. The calculation is based on 40,000 replications.}\label{fig:1}
\vspace{-0.1in}
\end{figure}

Theorem \ref{thm:limiting_dist3} shows that the distributions of $\lambda_{\min}(\hat\bSigma)$ and $\lambda_{\max}(\hat\bSigma)$ can be consistently estimated by those of their bootstrapped counterparts. Such an approximation is data adaptive, and does not require any parametric or semiparametric (such as elliptical distribution) assumption on the data generating scheme. Thus, such an inference procedure enjoys the distribution-free property \citep{kendall1977advanced}. In addition, the implementation is simple, and does not involve any tuning parameter.

Note that in Theorem \ref{thm:limiting_dist3}, we allow the dimension to slightly increase with the sample size. This is a more relaxed setting than that for the conventional bootstrap inference \citep{beran1985bootstrap}. Assumption \ref{cdt1} is required for the case that $d$ increases. However, when $d$ is fixed, this assumption can be easily relaxed. 

Finally, a comment on the scaling condition $d=o(n^{1/9})$ could be instructive. In detail, we aim to explore how sharp this condition is. For this, we generate $n=200, 500$, and $1,000$ data points from the multivariate standard Gaussian $\bX_1,\ldots,\bX_n\sim N_d(\zero,\Ib_d)$. We increase $d$ from 2, 5, 10, 20, 50, to 100. Figure \ref{fig:1} illustrates the exact distribution of $\lambda_{\max}(\hat\bSigma)$ (denoted as ``exact") and its multiplier-bootstrap and nonparametric bootstrap based counterparts (denoted as ``m-boots" and ``n-boots"). 

Figure \ref{fig:1} shows that, for $n=200$, the multiplier-bootstrap-based approach well approximates the exact distribution for $d$ not greater than 10. For $n=500$ and $1,000$, the dimension $d$ can be as large as $20$ and $50$ to ensure reasonable approximation results. This indicates that the dependence between $n$ and $d$ to guarantee efficient bootstrap approximation is almost linear, while we do need $d$ to be reasonably small compared to $n$. Also, Figure \ref{fig:1} shows that the nonparametric-bootstrap method leads to an extremely biased estimate of the distribution of $\lambda_{\max}(\hat\bSigma)$.

\section{Application II: two-sample Roy's largest root test}\label{sec:app2}

In multivariate analysis, tests for the equality of covariance matrices are of central interest \citep{anderson1958introduction}. High dimensionality brings  new challenges, and many existing methods cannot be used. 

Let $\bX_1,\ldots,\bX_{n}$ and $\bY_1,\ldots,\bY_{m}$ be $n$ and $m$ independent realizations of centered random vectors $\bX\in\reals^d$ and $\bY\in\reals^d$ with covariance matrices $\bSigma_1$ and $\bSigma_2$. In this section, we aim to test the hypothesis
\[
{\bf H_0}:\bSigma_1=\bSigma_2 \ \ \mbox{ versus } \ \ {\bf H_1}:\bSigma_1 \neq \bSigma_2
\]
under the scenario where $d$ is allowed to grow with $n$ and $m$.

We first briefly review the literature on testing ${\bf H_0}$ in high dimensions. \cite{johnstone2013roy} pointed out that the most common tests fall into two categories: the first is based on ``linear statistics" of the eigenvalues, and the second is based on extreme value statistics. In the high dimensional setting, \cite{chen2010tests}, \cite{li2012two}, \cite{cai2013optimal2}, among many others, have proposed tests based on linear statistics. Asymptotic normality is established even when $d/n$ tends to infinity. Initiated by \cite{jiang2004asymptotic}, another track of tests is developed based on extreme values of the entries of the sample covariance matrix. \cite{cai2013two} studied the problem of testing the equality of two unknown covariance matrices for possibly non-Gaussian data in the ``sparsity" scenario. Recently, \cite{chang2015bootstrap} proposed a bootstrap procedure to conduct inference for the test statistic considered in \cite{cai2013two} and relaxed the sparsity assumption.

Though significant progress has been made in this area, there has not been much research on Roy's largest root type tests \citep{roy1957some}, an important method in covariance testing. Such tests are built on extreme eigenvalues and are optimal against low-rank alternatives \citep{kritchman2009non}. Absence of the corresponding results in high dimensions is largely due to the uncommon behavior of extreme eigenvalues. Built on the results derived in Section \ref{sec:main}, we are able to fill this gap. This is done via exploiting a sparse-PCA-type thinking which is advocated by Iain Johnstone and many others \citep{johnstone2009consistency, cai2013sparse,ma2013sparse} in dealing with large random matrices. The techniques we developed here generalize those in \cite{cai2013two} and \cite{chang2015bootstrap}, and are of independent interest.

\subsection{Method}

The proposed test extends Roy's largest root test to high dimensions. In detail, let $\hat\bSigma_1$ and $\hat\bSigma_2$ be the sample covariance matrices given by
\[
\hat\bSigma_1 =\frac{1}{n}\sum_{i=1}^{n} \bX_i\bX_i^\T~~{\rm and}~~\hat\bSigma_2 =\frac{1}{m}\sum_{i=1}^{m} \bY_i \bY_i^\T. 
\]
Let $s$ be a prespecified parameter characterizing the sparsity level we wish to balance. To guarantee valid inference, we recommend $s$ to be chosen less than $\min(n,m)/10$. Recall in \eqref{cdt0} that $\mathbb{V}(s,d)$ represents the set of all $s$-sparse vectors in the unit sphere $\mathbb{S}^{d-1}$.
Let $\hat\cQ_{\bv}$ and $\hat\cQ_{\max}$ be the normalized rank-one projection and normalized $s$-sparse largest singular value of $\hat\bSigma_1-\hat\bSigma_2$:
\[
\hat\cQ_{\bv} = \sqrt{\frac{n+m}{nm}} \frac{\bv^\T(\hat\bSigma_1-\hat\bSigma_2)\bv}{\bv^\T(\hat\bSigma_1/n+\hat\bSigma_2/m)\bv}~~{\rm and}~~\hat\cQ_{\max} =\sup_{\bv\in\mathbb{V}(s,d)}| \cQ_{ \bv}|.
\]
The proposed test is multiplier-bootstrap-based. In detail, we define
\begin{align*}
\hat\cB_{ \bv} =\sqrt{\frac{n+m}{nm}} \frac{\bv^\T \{ \sum_{i=1}^{n}\xi_i(\bX_i\bX_i^\T-\hat\bSigma_1)/n -  \sum_{i=1}^{m}\eta_i(\bY_i\bY_i^\T-\hat\bSigma_2)/ m \} \bv}{\bv^\T(\hat\bSigma_1/n+\hat\bSigma_2/m)\bv}
\end{align*}
and
\[
\hat\cB_{\max}   = \sup_{\bv\in\mathbb{V}(s,d)} | \hat\cB_{ \bv}|,
\]
where $\xi_1,\ldots,\xi_{n},  \eta_{1},\ldots, \eta_{m}$ are independent standard Gaussian random variables that are independent of $\{\bX_i\}_{i=1}^{n}$ and $\{\bY_i\}_{i=1}^{m}$. Let $\cB_1^*,\ldots, \cB_N^*$ be $N$ realizations of $\hat \cB_{\max}$ (via fixing the data and changing $\xi_1,\ldots,\xi_{n}, \eta_{1},\ldots, \eta_{m}$) for some large enough $N$\footnote{In the sequel, for ease of presentation, we focus on the ideal case that we know the exact conditional distribution of $\hat \cB_{\max}$ given the data. This is equivalent to setting $N$ infinitely large. In practice, the accuracy of bootstrap by setting a finitely large $N$ is guaranteed by the Dvoretzky-Kiefer-Wolfowitz inequality \citep{dvoretzky1956asymptotic}.}. Let $q_{\alpha}$ be the corresponding $(1-\alpha)$-th quantile. The proposed test is
\begin{align}\label{eq:hantest}
T_{\alpha}  = \ind(\hat\cQ_{\max} \geq  q_{\alpha}),
\end{align}
and we reject ${\bf H_0}$ whenever $T_{\alpha}=1$.

\begin{remark} {\rm
Computing the extreme eigenvalues of large covariance matrices under sparsity constraint involves a combinatorial optimization and is NP-complete in general. Several computationally feasible methods based on the recent developments in the sparse PCA literature can be used to compute $\hat\cQ_{\max}$ approximately. The theoretical guarantees, however, remain unclear. For example, a greedy search to shrink the candidate set \citep{moghaddam2005spectral}, followed by a second-step brutal search, may work well in practice. Recently, \cite{berthet2013spca} proposed a computationally efficient method using convex relaxations to compute their sparse eigenvalue statistic for sparse principal component testing. It is interesting to investigate whether their method can be adapted to deal with the current problem. We leave this to future work.}
\end{remark}

\begin{remark}{\rm
A ``non-normalized" version of the test $T_{\alpha}$, based on the restricted spectral gap $\sup_{\bv\in{\mathbb{V}}(s)}|\bv^\T(\hat\bSigma_1-\hat\bSigma_2)\bv|$, can be similarly defined and calculated using the truncated power method \citep{yuan2013truncated}. 
However, boundeness on the restricted eigenvalue $\lambda_{\max, s}(\bSigma_1 )$ is required for the validity of the non-normalized test.}
\end{remark}

\subsection{Theory}

This section provides the theoretical properties of $T_{\alpha}$ in \eqref{eq:hantest}. First, we show that the size of the test is well controlled. Secondly, we study the power and prove the minimax optimality for the proposed test against ``low-rank" alternatives.

To ensure the size consistency of $T_{\alpha}$, we require the following two assumptions on the data generating scheme. They are analogous to those in Section \ref{sec:main}. Of note, we do not require $\bX\stackrel{\sf d}{=}\bY$.

\begin{assumption}\label{cdt5}
There exist random vectors $\bU_1,\bU_2 \in\reals^d$ satisfying $\E ( \bU_i )=\zero$ and $\E ( \bU_i\bU_i^{\T} )=\Ib_d$ for $i=1,2$, such that
\begin{align*}
\bX=\bSigma_1^{1/2}\bU_1, ~~\bY=\bSigma_2^{1/2}\bU_2, ~~L_1 :=\max\bigg( \sup_{\bv\in\mathbb{S}^{d-1}}\|\bv^{\T}\bU_1\|_{\psi_2}, \sup_{\bv\in\mathbb{S}^{d-1}}\|\bv^{\T}\bU_2\|_{\psi_2}\bigg)<\infty .
\end{align*}
\end{assumption}
\begin{assumption}\label{cdt6}
$\{\bX_i\}_{i=1}^{n}$ and $\{\bY_i\}_{i=1}^{m}$ are independent realizations of $\bX$ and $\bY$, respectively. Moreover, the sample sizes are comparable, i.e., $n\asymp m$.
\end{assumption}

The following result provides the theoretical guarantee for the validity of the multiplier bootstrap test $T_{\alpha}$. 

\begin{theorem}\label{thm:testing} Let Assumptions \ref{cdt5} and \ref{cdt6} be satisfied. Under the null hypothesis ${\bf H_0}$, we have
\begin{align}\label{eq:han1}
\sup_{t\geq 0}\bigg| \P( \hat\cQ_{\max} \leq t) - \P( \hat\cB_{\max} \leq t \, | \bX_1,\ldots,\bX_{n},\bY_1,\ldots,\bY_{m} )  \bigg| = o_{\mathbb{P}}(1) 
\end{align}
whenever $s\log\{ \gamma_s(\bSigma_1) ed/s\} \vee s\log n=o(n^{1/9})$. Moreover, as $n, m \to \infty$,
\[
\P_{\bf H_0}(T_{\alpha}=1)=\alpha+o(1)
\]
uniformly in $0<\alpha<1$.
\end{theorem}

Next we analyze the power of the test $T_{\alpha}$ and show that it is minimax rate-optimal. For this, we consider the alternative class of matrices
\begin{align*}
\mathbb{M}(\lambda) := \bigg\{(\Mb_1,\Mb_2)\in& \reals^{d\times d}:  \, \lambda_{\max, s}( \Mb_\ell )/\lambda_{\min, s}( \Mb_\ell  ) \leq C ,\, \ell = 1, 2,  \\
&\sqrt{\frac{n+m}{nm}} \sup_{\bv\in\mathbb{V}(s,d)}\bigg|\frac{\bv^\T(\Mb_1-\Mb_2)\bv}{\bv^\T(\Mb_1/n+\Mb_2/m)\bv}\bigg|\geq \lambda\sqrt{s\log (ed/s)} \bigg\}
\end{align*}
for some constant $C>0$ independent of $(n,m,d)$. The following two theorems illustrate the power and minimax lower bound in differentiating two covariance matrices within some matrix set $\mathbb{M}(\lambda)$.

\begin{theorem}[Power analysis]\label{thm:power} Suppose that Assumptions \ref{cdt5} and \ref{cdt6} hold. Further assume that $(n,m,d,s)$ satisfies $s\log (ed/s)\log n=o(n)$. Then, for all sufficiently large $\lambda>0$, we have
\[
\inf_{(\bSigma_1,\bSigma_2)\in\mathbb{M}(\lambda)}\P_{(\bSigma_1,\bSigma_2)}(T_{\alpha}=1) \to 1, \ \ \mbox{ as } n \to \infty,
\]
where $\P_{(\bSigma_1,\bSigma_2)}$ represents the joint distribution of $\bX_1,\ldots,\bX_n,\bY_1,\ldots,\bY_m$ 
with covariance matrices $\bSigma_1$ and $\bSigma_2$.
\end{theorem}

\begin{theorem}[Minimax lower bound]\label{thm:lower_bound}
Assume that the conditions in Theorem~\ref{thm:power} hold. Then, for all sufficiently small $\lambda>0$, we have
\begin{align*}
\inf_{\Phi_\alpha \in \cT_{\alpha}}\sup_{(\bSigma_1,\bSigma_2)\in\mathbb{M}(\lambda)}\P_{(\bSigma_1,\bSigma_2)}( \Phi_\alpha =0)\geq 1-\alpha-o(1), \ \ \mbox{ as } n \to \infty,
\end{align*}
where $\cT_{\alpha} = \{ \Phi_\alpha  : \P _{\bf H_0}(\Phi_\alpha = 1) \leq \alpha \}$ denotes the class of all $\alpha$-level tests.
\end{theorem}

\begin{remark}{\rm 
The scaling condition, $s\log (ed/s)\log n=o(n)$, in Theorems \ref{thm:power} and \ref{thm:lower_bound}, is weaker than the corresponding one in Theorem \ref{thm:testing}.  Accordingly, combining Theorems \ref{thm:testing}, \ref{thm:power}, and \ref{thm:lower_bound}, under Assumptions \ref{cdt5} and \ref{cdt6}, the scaling condition in Theorem \ref{thm:testing}, and the boundedness assumption on $\lambda_{\max, s}( \bSigma_\ell )/\lambda_{\min, s}(\bSigma_\ell)$ for $\ell=1,2$, the proposed test is minimax rate-optimal. }
\end{remark}

\begin{remark}{\rm
The $\log n$ term in the scaling condition, $s\log (ed/s)\log n=o(n)$, is required for handling the multiplier-bootstrap-based statistic $\hat\cB_{\max}$, which involves the cubes of sub-Gaussian random variables. We tackle it via a truncation argument, which is also exploited in \cite{cai2011adaptive} and \cite{cai2013two}. Similar scaling conditions are also posed therein.}
\end{remark}

\subsection{Empirical results}

In this section, we compare the numerical performance of the proposed approach with two existing ones. Specifically, we consider the following tests.
\begin{itemize}
\item
{\sf HXZ3}: the proposed covariance test with $s=3$;
\item
{\sf HXZ5}: the proposed covariance test with $s=5$;
\item
{\sf HXZ10}: the proposed covariance test with $s=10$;
\item
{\sf CZZW}: the $L_\infty$-type bootstrap test proposed in \cite{chang2015bootstrap};
\item
{\sf LC}: the $L_2$-type covariance test proposed in \cite{li2012two}.
\end{itemize}

These five tests represent the three notable tracks in covariance testing problems. Specifically, our proposed tests are more sensitive to low-rank alternatives, while {\sf CZZW} and {\sf LC} are more sensitive to elementwise changes of the covariance matrix. To implement the proposed tests and that of {\sf CZZW}, we take the bootstrap sample size to be 1,000. 

We consider three settings for the structure of $\bSigma_1$. Let $\Ob$ be a diagonal matrix with its diagonals generated from a uniform distribution ${\rm Unif}(0.5,1.5)$. Set $\bSigma_1=\Ob\bSigma^*\Ob$, where $\bSigma^*=[\bSigma^*_{jk}]\in\reals^{d\times d}$ is specified as follows:
\begin{itemize}
\item {\bf long range:} $\bSigma^*_{jk}=\ind(j=k)+0.5\ind(j\ne k)$;
\item {\bf short range:} $\bSigma^*_{jk}=0.1^{|j-k|}$;
\item {\bf isotropic:} $\bSigma^*=\Ib_d$. 
\end{itemize}

To compare the power, we consider three types of alternatives:
\begin{itemize}
\item {\bf alternative 1:} $\bSigma_2=\bSigma_1+c_1\cdot\bv_S\bv_S^\T$, where $\bv_S$ is sparse with the support size $5$ and non-zero entries all equal to $1/\sqrt{5}$;
\item {\bf alternative 2:} $\bSigma_2=\bSigma_1+c_2\Db^{(1)}$, where $\Db^{(1)}$ satisfies $\Db_{12}^{(1)}=\Db_{21}^{(1)}=1$ and has 0 elsewhere;
\item {\bf alternative 3:} $\bSigma_2=(\Ib_d+c_3\Db^{(2)})^\T\bSigma_1(\Ib_d+c_3\Db^{(2)})$, where $\Db^{(2)}$ satisfies $\Db^{(2)}_{j+1,j}=1$ for $j=1,\ldots,d-1$ and has 0 elsewhere.  
\end{itemize}

We also generate $n=m$ independent samples from $N_d(\zero,\bSigma_1)$ and $N_d(\zero,\bSigma_2)$ separately. Here we set $n=20, 500, 1,000$ and $d =40, 100$. The values of $c_1, c_2, c_3$ are specified as follows.
\begin{itemize}
\item {\bf long range:} Set $(c_1,c_2,c_3)\!=\!(0.9,0.35,0.1)$ for $d \!=\!40$, and set $(c_1,c_2,c_3)\!=\!(1.1,0.4,0.1)$ for $d \!=\!100$;
\item {\bf short range:} Set $(c_1,c_2,c_3)\!=\!(0.7,0.4,0.1)$ for $d \!=\!40$, and set $(c_1,c_2,c_3)\!=\!(0.85,0.4,0.1)$ for $d \!=\!100$;
\item {\bf isotropic:} Set $(c_1,c_2,c_3)=(0.8, 0.3, 0.12)$ for $d =40$, and set $(c_1,c_2,c_3)\!=\!(0.85,0.3,\\0.1)$ for $d \!=\!100$.
\end{itemize}

We repeat the simulation 5,000 times. Tables \ref{tab:1}-\ref{tab:3} illustrate the size and powers (corresponding to the three alternatives) for each combination of $n$ and $d$. There are several noteworthy observations. First, empirical sizes of the proposed tests and {\sf CZZW} are well controlled, while the size of {\sf LC} is inflated in the ``long range" setting. This is as expected since the ``long range" dependence violates the assumptions in \cite{li2012two}. Secondly, regarding the empirical powers, our proposed tests outperform those of {\sf CZZW} and {\sf LC} under alternative 1, which is in line with Theorems \ref{thm:testing} and \ref{thm:lower_bound}; {\sf CZZW} outperforms the others under alternative 2, and {\sf LC} outperforms the others under alternative 3, as expected. Finally, we see that {\sf HXZ5}  performs the best on average in alternative 1, 
while {\sf HXZ3} and {\sf HXZ10} perform similarly to {\sf HXZ5}, and all outperform {\sf CZZW} and {\sf LC}. Hence, for testing against low-rank alternatives, our proposed test is more favorable and the choice of $s$ is rather flexible as long as it remains at a moderate size.







{
\renewcommand{\tabcolsep}{5pt}
\renewcommand{\arraystretch}{1.00}
\begin{table}[t]
\footnotesize
\caption{(Long Range) Comparison of the five competing tests under the null and alternatives 1 to 3 when $n=200, 500, 1,000$ and $d=40, 100$. The results are computed based on 5,000 replications. }
\begin{center}
\begin{tabular}{cccccccccccc}
\toprule
 $d$ & $n$ &  {\sf HXZ3} &{\sf HXZ5} &{\sf HXZ10} & {\sf CZZW} & {\sf LC}  & {\sf HXZ3} &{\sf HXZ5} &{\sf HXZ10} & {\sf CZZW} & {\sf LC}  \\
 \cmidrule(r){3-7}  \cmidrule(r){8-12}
 &&  \multicolumn{5}{c}{empirical size} & \multicolumn{5}{c}{empirical power (alternative 1)} \\
\midrule
40 & 200 & 0.031 & 0.022 & 0.000 & 0.073 & 0.109 & 0.230 & \textbf{0.342} & 0.284 & 0.108 & 0.102 \\ 
     & 500 & 0.046 & 0.036 & 0.013 & 0.055 & 0.095 & 0.802 & 0.911 & \textbf{0.931} & 0.358 & 0.146 \\ 
     & 1000 & 0.045 & 0.042 & 0.027 & 0.040 & 0.085 & 0.990 & 0.997 & \textbf{1.000} & 0.756 & 0.150 \\ 
\midrule
 100 & 200 & 0.018 & 0.007 & 0.000 & 0.088 & 0.111 & 0.120 & \textbf{0.184} & 0.064 & 0.089 & 0.105 \\ 
     & 500 & 0.044 & 0.016 & 0.003 & 0.035 & 0.093 & 0.824 & 0.878 & \textbf{0.894} & 0.340 & 0.106 \\ 
     & 1000 & 0.043 & 0.029 & 0.018 & 0.048 & 0.108 & 0.993 & \textbf{0.996} & 0.995 & 0.770 & 0.122 \\ 
 \midrule
 &&  \multicolumn{5}{c}{empirical power (alternative 2)} & \multicolumn{5}{c}{empirical power (alternative 3)} \\
\midrule
40 & 200 & 0.121 & 0.202 & \textbf{0.347} & 0.220 & 0.106 & 0.068 & 0.036 & 0.007 & 0.269 & \textbf{0.398} \\ 
     & 500 & 0.141 & 0.210 & 0.386 & \textbf{0.788} & 0.096 & 0.211 & 0.123 & 0.078 & 0.503 & \textbf{0.654} \\ 
     & 1000 & 0.120 & 0.203 & 0.431 & \textbf{0.999} & 0.089 & 0.537 & 0.505 & 0.353 & 0.854 & \textbf{0.923} \\ 
\midrule
 100 & 200 & 0.042 & 0.071 & 0.147 & \textbf{0.252} & 0.123 & 0.044 & 0.003 & 0.000 & 0.200 & \textbf{0.369} \\ 
     & 500 & 0.079 & 0.097 & 0.158 & \textbf{0.847} & 0.118 & 0.170 & 0.100 & 0.028 & 0.550 & \textbf{0.671} \\ 
     & 1000 & 0.052 & 0.095 & 0.185 & \textbf{1.000} & 0.101 & 0.525 & 0.418 & 0.255 & 0.840 & \textbf{0.912 }\\  
\bottomrule
 \end{tabular}
 \end{center}\label{tab:1}
 \end{table}
}

{
\renewcommand{\tabcolsep}{5pt}
\renewcommand{\arraystretch}{1.00}
\begin{table}[tp!]
\footnotesize
\caption{(Short Range) Comparison of the five competing tests under the null and alternatives 1 to 3 when $n = 200, 500, 1,000$ and $d=40,100$. The results are computed based on 5,000 replications.}
\begin{center}
\begin{tabular}{cccccccccccc}
\toprule
 $d$ & $n$ &  {\sf HXZ3} &{\sf HXZ5} &{\sf HXZ10} & {\sf CZZW} & {\sf LC}  & {\sf HXZ3} &{\sf HXZ5} &{\sf HXZ10} & {\sf CZZW} & {\sf LC}  \\
 \cmidrule(r){3-7}  \cmidrule(r){8-12}
 &&  \multicolumn{5}{c}{empirical size} & \multicolumn{5}{c}{empirical power (alternative 1)} \\
\midrule
40 & 200 & 0.023 & 0.014 & 0.001 & 0.039 & 0.030 & 0.075 & 0.053 & 0.002 & 0.083 & \textbf{0.160} \\ 
     & 500 & 0.036 & 0.028 & 0.024 & 0.066 & 0.024 & 0.566 & \textbf{0.570} & 0.465 & 0.390 & 0.213 \\ 
     & 1000 & 0.048 & 0.028 & 0.032 & 0.050 & 0.020 & 0.936 & 0.932 & \textbf{0.943} & 0.848 & 0.426 \\ 
\midrule
 100 & 200 & 0.007 & 0.001 & 0.000 & 0.039 & 0.127 & 0.021 & 0.035 & 0.001 & 0.033 & \textbf{0.209} \\ 
     & 500 & 0.030 & 0.020 & 0.005 & 0.044 & 0.034 & \textbf{0.362} & 0.333 & 0.186 & 0.223 & 0.206 \\ 
     & 1000 & 0.036 & 0.030 & 0.010 & 0.047 & 0.026 & 0.894 & \textbf{0.901} & 0.898 & 0.756 & 0.196 \\ 
 \midrule
 &&  \multicolumn{5}{c}{empirical power (alternative 2)} & \multicolumn{5}{c}{empirical power (alternative 3)} \\
\midrule
40 & 200 & 0.041 & 0.020 & 0.004 & \textbf{0.179} & 0.070 & 0.027 & 0.009 & 0.002 & 0.154 & \textbf{0.364} \\ 
     & 500 & 0.104 & 0.132 & 0.081 & \textbf{0.701} & 0.021 & 0.183 & 0.192 & 0.106 & 0.470 & \textbf{0.730} \\ 
     & 1000 & 0.123 & 0.181 & 0.294 & \textbf{0.994} & 0.046 & 0.486 & 0.462 & 0.513 & 0.975 & \textbf{0.993} \\ 
\midrule
 100 & 200 & 0.028 & 0.020 & 0.002 & 0.088 & \textbf{0.159} & 0.037 & 0.025 & 0.001 & 0.075 & \textbf{0.467} \\ 
     & 500 & 0.052 & 0.037 & 0.013 & \textbf{0.580} & 0.161 & 0.097 & 0.087 & 0.025 & 0.258 & \textbf{0.756} \\ 
     & 1000 & 0.080 & 0.098 & 0.107 & \textbf{0.981} & 0.169 & 0.199 & 0.215 & 0.168 & 0.759 & \textbf{0.990} \\  
\bottomrule
 \end{tabular}
 \end{center}\label{tab:2}
 \end{table}
}

{
\renewcommand{\tabcolsep}{5pt}
\renewcommand{\arraystretch}{1.00}
\begin{table}[t]
\footnotesize
\caption{(Isotropic) Comparison of the five competing tests under the null and alternatives 1 to 3 when $n=200, 500, 1,000$ and $d=40, 100$. The results are computed based on 5,000 replications. }
\begin{center}
\begin{tabular}{cccccccccccc}
\toprule
 $d $ & $n$ &  {\sf HXZ3} &{\sf HXZ5} &{\sf HXZ10} & {\sf CZZW} & {\sf LC}  & {\sf HXZ3} &{\sf HXZ5} &{\sf HXZ10} & {\sf CZZW} & {\sf LC}  \\
 \cmidrule(r){3-7}  \cmidrule(r){8-12}
 &&  \multicolumn{5}{c}{empirical size} & \multicolumn{5}{c}{empirical power (alternative 1)} \\
\midrule
40 & 200 & 0.029 & 0.013 & 0.000 & 0.061 & 0.042 & 0.053 & 0.039 & 0.009 & 0.076 & \textbf{0.159} \\ 
     & 500 & 0.044 & 0.015 & 0.015 & 0.052 & 0.053 & 0.435 & \textbf{0.500} & 0.443 & 0.253 & 0.410 \\ 
     & 1000 & 0.042 & 0.044 & 0.029 & 0.050 & 0.061 & 0.872 & 0.920 & \textbf{0.939} & 0.763 & 0.854 \\ 
\midrule
 100 & 200 & 0.015 & 0.001 & 0.000 & 0.036 & 0.058 & 0.035 & 0.005 & 0.000 & 0.062 & \textbf{0.099} \\ 
     & 500 & 0.039 & 0.016 & 0.006 & 0.060 & 0.040 & 0.430 & \textbf{0.455} & 0.322 & 0.250 & 0.238 \\ 
     & 1000 & 0.039 & 0.032 & 0.017 & 0.033 & 0.065 & 0.900 & 0.904 & \textbf{0.908} & 0.806 & 0.554 \\ 
 \midrule
 &&  \multicolumn{5}{c}{empirical power (alternative 2)} & \multicolumn{5}{c}{empirical power (alternative 3)} \\
\midrule
 40 & 200 & 0.038 & 0.017 & 0.006 & \textbf{0.461} & 0.097 & 0.025 & 0.007 & 0.001 & 0.091 & \textbf{0.269} \\ 
     & 500 & 0.120 & 0.169 & 0.198 & \textbf{0.990} & 0.288 & 0.075 & 0.066 & 0.070 & 0.289 & \textbf{0.729} \\ 
     & 1000 & 0.141 & 0.228 & 0.421 & \textbf{1.000} & 0.604 & 0.213 & 0.223 & 0.223 & 0.787 & \textbf{0.994} \\ 
\midrule
 100 & 200 & 0.030 & 0.010 & 0.000 & \textbf{0.312} & 0.073 & 0.009 & 0.005 & 0.000 & 0.079 & \textbf{0.253} \\ 
     & 500 & 0.065 & 0.063 & 0.026 & \textbf{0.982} & 0.106 & 0.043 & 0.044 & 0.017 & 0.223 & \textbf{0.769} \\ 
     & 1000 & 0.079 & 0.135 & 0.189 & \textbf{1.000} & 0.207 & 0.160 & 0.146 & 0.079 & 0.820 & \textbf{0.999} \\  
\bottomrule
 \end{tabular}
 \end{center}\label{tab:3}
 \end{table}
}

\section{Discussion}\label{sec:discussion}

Spectral analysis for large random matrices has a long history and maintains one of the most active research areas in statistics. Recent advances include the {discovery of the Tracy-Widom law, an important family of distributions that quantifies the fluctuation of sample eigenvalues.} A vast literature follows. However, more questions are raised than answered. In particular, no result has been promised for extensions to non-Gaussian distributions with a nontrivial covariance structure. This paper fills this long-standing gap from a new perspective grown in the literature of extreme value theory. The obtained results prove to work in many cases which for a long time are known to be challenging to deal with.

Very recently, \cite{fan2015asymptotics} studied asymptotic behaviors of sample covariance eigenvalues under a pervasive assumption, that is, the largest eigenvalue grows quickly with the dimension. Under this assumption, they proved the asymptotic normality for the sample eigenvalues. In comparison, our results are built on the normalized covariance matrix and are obtained in the settings where the signals are not too strong. A natural question arises that whether a phase transition phenomenon occurs when signals change from weak to strong. In particular, how do the asymptotic distributions of sample eigenvalues change with the growing magnitudes of extreme eigenvalues? We conjecture that this problem may be related to the normal mean problem in extreme value theory, and leave that question for future research.

\section{Proofs}\label{sec:proof}

This section contains the proofs of the results in this paper.

\subsection{Proof of the main results}\label{sec:mainproof}

\subsubsection{Proof of Theorem \ref{thm:limiting_dist}}
We first give an outline of the proof, which consists of three main steps. (i) In the first step, we approximate $\hat Q_{\max}$, the supremum over a continuous function space induced by $\mathbb{V}(s,d)$, by the maximum over a discrete function space induced by $\mathbb{N}_{\epsilon_1}$, for $\epsilon_1$ as in Theorem \ref{thm:limiting_dist}. (ii) In the second step, we show that the above discretized version of $\hat{Q}_{\max}$ over $\mathbb{N}_{\epsilon_1}$ converges weakly to the maximum of a Gaussian sequence. (iii) Lastly, we employ the anti-concentration inequality (Lemma \ref{lem:anticoncentration}) to complete the proof.

{\bf Step I.} Let $\epsilon\in (0,1)$ be an arbitrary number. We first employ the following lemma to connect the supremum over a continuous function space induced by $\mathbb{V}(s,d)$ to the maximum over a discrete function space induced by $\mathbb{N}_{\epsilon}$.
\begin{lemma}\label{lem:discret}
There exists an $\epsilon$-net $\mathbb{N}^0_{\epsilon}$ of $\mathbb{V}(s,d)$ equipped with the Euclidean metric satisfying that $\log p^0_{\epsilon} = \log |\mathbb{N}^0_{\epsilon}| \lesssim s\log\frac{ed}{\epsilon s}$. Further, for any $\epsilon$-net $\mathbb{N}_{\epsilon}$ of $\mathbb{V}(s,d)$, we have
\begin{align*}
\hat{Q}_{\max} \leq2\gamma_s \epsilon\cdot \hat{Q}_{\max} +M_{\max,\epsilon},
\end{align*}
where
\begin{align}\label{eq:def_Mn}
M_{\max,\epsilon} :=\max_{\bv\in\mathbb{N}_{\epsilon}}|Q_{ \bv}|=\max_{\bv\in\mathbb{N}_{\epsilon}}\bigg|\frac{\sqrt{n}~\bv^{\T}(\hat\bSigma-\bSigma)\bv}{\bv^{\T}\bSigma\bv}\bigg|.
\end{align}
\end{lemma}

\medskip
Lemma \ref{lem:discret} and the fact $\hat{Q}_{\max}\geq M_{\max,\epsilon}$ yield that, for any $\epsilon\in(0,1)$,
\begin{align}\label{eq:step11}
|\hat{Q}_{\max}-M_{\max,\epsilon} |\leq2 \gamma_s \epsilon\cdot \hat{Q}_{\max}.
\end{align}
Next we use Lemma \ref{lem:connecting} below to bound $\hat{Q}_{\max}$. Note that, by Lemma \ref{lem:orlicz norm}, we have for any $\bv\in\mathbb{S}^{d-1}$ that $
\|\bv^{\T}\bU\bU^{\T}\bv-1\|_{\psi_1}\leq\|(\bv^{\T}\bU)^2\|_{\psi_1}+1=\|\bv^{\T}\bU\|_{\psi_2}^2+1$. Taking maximum over $\bv\in\mathbb{S}^{d-1}$ on both sides yields
\begin{align*}
 \sup_{\bv\in\mathbb{S}^{d-1}}\|\bv^{\T}\bU\bU^{\T}\bv-1\|_{\psi_1} \leq K_1^2 +1 = K_2
\end{align*}
\begin{lemma}\label{lem:connecting}
For any $t>0$, we have
\begin{align*}
\P\bigg[ \hat{Q}_{\max} \leq C_{12} K_2 \,   \gamma^{1/2}_n(s,d) + K_2 \max\bigg\{2  \sqrt{2t}, C_{11}   \gamma_n(s,d) \frac{t}{\sqrt{n}}\bigg\}\bigg]
\geq1-4e^{-t},
\end{align*}
where $C_{11}, C_{12}>0$ are absolute constants, $K_2=K_1^2+1$ and $\gamma_n(s,d)=s\log(\gamma_s \cdot ed/s)\vee s\log n$ are as in Theorem \ref{thm:limiting_dist}. 
\end{lemma}

\medskip
Using Lemma \ref{lem:connecting}, it follows from \eqref{eq:step11} that for any $t>0$ and $\epsilon\in(0,1)$,
\begin{align*}
\P\bigg(|\hat{Q}_{\max}\!-\!M_{\max,\epsilon}|  \leq C K_2\, \gamma_s \epsilon  \bigg[  \gamma^{1/2}_n(s,d)  \!+\!   \max\bigg\{   \sqrt{t}  ,   \gamma_n(s,d)\frac{t}{\sqrt{n}} \bigg\} \bigg] \bigg)\geq1\!-\!4e^{-t}.
\end{align*}
Taking $\epsilon=\epsilon_1 =(n\gamma_s )^{-1}$ and $t=\log n$, we deduce that
\begin{align}\label{eq:step1}
\P\bigg( \bigg|\hat{Q}_{\max}\!-\!\max_{\bv\in\mathbb{N}_{\epsilon_1}}|Q_{\bv}| \bigg| \leq  C K_2 \bigg[  \frac{ \gamma^{1/2}_n(s,d) }{n}\!+\!   \max \bigg\{ \frac{\sqrt{ \log n}}{n},   \gamma_n(s,d)\frac{\log n}{n^{3/2}}\bigg\} \bigg] \bigg)\notag \\
 \geq1-\frac{4}{n}.
\end{align}

\medskip
{\bf Step II.} For any $\epsilon\in(0,1)$, write $\mathbb{N}_{\epsilon} =\{\bv_j:j=1,\ldots,p_{\epsilon}\}$ for the $\epsilon$-net $\mathbb{N}_{\epsilon}$ constructed in {\bf Step I}, and recall that for $i=1,\ldots,n$ and $j=1,\ldots,p_{\epsilon}$,
\begin{align*}
R_{ij} =\frac{\bv_j^{\T}(\bX_i\bX_i^{\T}-\bSigma)\bv_j}{\bv_j^{\T}\bSigma\bv_j}.
\end{align*}
It follows from \eqref{eq:def_Mn} that
\begin{align*}
M_{\max,\epsilon}=\max_{\bv\in\mathbb{N}_{\epsilon}}|Q_{ \bv}|=\max_{1\leq j\leq p_{\epsilon}}\bigg|\frac{1}{\sqrt{n}}\sum_{i=1}^n R_{ij}\bigg|.
\end{align*}
The following lemma gives a Gaussian coupling inequality for $M_{\max,\epsilon}$.

\begin{lemma}\label{lem:coupling}
For every $\delta>0$, we have
\begin{align*}
&\P\bigg(\bigg| M_{\max,\epsilon} - \max_{1\leq j\leq p_{\epsilon}}|G_j|\bigg|\geq16\delta\bigg)\notag\\
\lesssim & \, K_2^2 \frac{ \sqrt{ \log (2p_{\epsilon} )  }\log(2p_{\epsilon}\vee n)}{\delta^2 n^{1/2} }+ K_0^2K_2^2  \frac{\log (2p_{\epsilon} ) \{ \log(np_{\epsilon}+1)\}^2 \log(2p_{\epsilon}\vee n)}{\delta^2n}   \notag \\
& \quad +K_2^3 \frac{ \{ \log(2p_{\epsilon}\vee n)\}^2}{\delta^3n^{1/2} }  +  K_0^3   K_2^3 \frac{(\log p_{\epsilon})\{ \log(np_{\epsilon}+1) \}^3 \{ \log(2p_{\epsilon}\vee n)\}^2}{\delta^3 n^{3/2}} \notag \\
& \quad + K_0^4K_2^4 \frac{\{ \log(p_{\epsilon}+1)\}^4 \{ \log(2p_{\epsilon}\vee n)\}^3}{\delta^4n}+\frac{\log n}{n},
\end{align*}
where $M_{\max,\epsilon}$ is as in \eqref{eq:def_Mn}, $K_0$ is the constant in Lemma \ref{lem:maximal_ineq} with $\alpha=1$ and $K_2=K_1^2+1$.
\end{lemma}

In view of Lemma \ref{lem:coupling}, by taking $\epsilon=\epsilon_1$ we have
\begin{align*}
&  P\bigg(\bigg|\max_{\bv\in\mathbb{N}_{\epsilon_1}}|Q_{ \bv}|\!-\!\max_{1\leq j\leq p_{\epsilon_1}}|G_j|\bigg|\geq16\delta\bigg)  \\ 
  \lesssim& K_2^3 \frac{ \{ \log(2p_{\epsilon_1}\vee n) \}^2}{\delta^3n^{1/2}}
 + K_0^4K_2^4 \frac{ \{ \log(p_{\epsilon_1}\!+1) \}^4 \{ \log(2p_{\epsilon_1}\!\vee n)\}^3}{\delta^4n},
\end{align*}
where $p_{\epsilon_1} =|\mathbb{N}_{\epsilon_1}|$. Putting $\{G_{\bv} \}_{\bv \in \mathbb{N}_{\epsilon_1}}= \{ G_j \}_{j=1}^{p_{\epsilon_1}}$, we have
\begin{align}\label{eq:step2}
P\bigg(\bigg|\max_{\bv\in\mathbb{N}_{\epsilon_1}}|Q_{ \bv}|-\max_{\bv\in\mathbb{N}_{\epsilon_1}}|G_{\bv}|\bigg|\geq16\delta\bigg)\lesssim K_2^3\frac{ (\log p_{\epsilon_1})^2 }{\delta^3n^{1/2}}+ K_0^4K_2^4  \frac{ (\log p_{\epsilon_1} )^7 }{\delta^4  n} . 
\end{align}
Without loss of generality, assume that $\log p_{\epsilon_1} \geq \gamma_n(s,d)$ (the case when $\log p_{\epsilon_1}<\gamma_n(s,d)$ can be similarly dealt with by replacing all $\log p_{\epsilon_1}$ below by $\gamma_n(s,d)$). Combining \eqref{eq:step1} and \eqref{eq:step2}, we have
\begin{align}\label{eq:step2fin}
\P\bigg\{ \bigg|\hat{Q}_{\max}-\max_{\bv\in\mathbb{N}_{\epsilon_1}}|G_{\bv}|\bigg|\geq16\delta+CK_2\bigg( \! \frac{ \sqrt{\log p_{\epsilon_1}} }{n}+\log p_{\epsilon_1} \frac{\log n}{n^{3/2}}\bigg) \bigg\}\notag \\
 \lesssim  K_2^3  \frac{(\log p_{\epsilon_1})^2 }{\delta^3n^{1/2}}+K_0^4K_2^4\frac{ (\log p_{\epsilon_1})^7 }{\delta^4n},
\end{align}
where $C>0$ is an absolute constant.

\medskip
{\bf Step III.} Taking $\delta=\frac{K_0^4 (\log p_{\epsilon_1})^{5/8} }{K_2 \, n^{1/8}}$, we deduce from \eqref{eq:step2fin} that there exists an absolute positive constant $C$ such that 
\begin{align}\label{eq:step31}
\P\bigg\{ \bigg|\hat{Q}_{\max}-\max_{\bv\in\mathbb{N}_{\epsilon_1}}|G_{\bv}|\bigg|\geq C K_0^4 \frac{ (\log p_{\epsilon_1})^{5/8} }{K_2 \, n^{1/8}}\bigg\} \lesssim K_2^2 \frac{  (\log p_{\epsilon_1})^{1/8} }{n^{1/8}}+K_2^2 \frac{ (\log p_{\epsilon_1})^{9/2} }{n^{1/2}}.
\end{align}
By Lemma \ref{lem:anticoncentration}, we have
\begin{align}\label{eq:step32}
\sup_{x\in\reals}\P\bigg\{ \bigg|\max_{\bv\in\mathbb{N}_{\epsilon_1}}|G_{\bv}|-x\bigg|\leq C_{13 } K_0^4 \frac{(\log p_{\epsilon_1})^{5/8} }{K_2 \, n^{1/8}}\bigg\} \lesssim K_2^2 \frac{(\log p_{\epsilon_1})^{9/8}  }{n^{1/8}} 
\end{align}
for some absolute constant $C_{13} >0$. Note that, for every $t>0$ and $\eta>0$,
\begin{align*}
& \bigg|\P(\hat{Q}_{\max}\leq t)-\P\bigg(\max_{\bv\in\mathbb{N}_{\epsilon_1}}|G_{\bv}|\leq t\bigg)\bigg| \\ 
 \leq&\P\bigg(\max_{\bv\in\mathbb{N}_{\epsilon_1}}|G_{\bv}|\in[t-\eta,t+\eta]\bigg)+\P\bigg( \bigg| \hat{Q}_{\max}-\max_{\bv\in\mathbb{N}_{\epsilon_1}}|G_{\bv}| \bigg|>\eta\bigg).
\end{align*}
Taking $\eta = C_{13}\frac{K_0^4 (\log p_{\epsilon_1})^{5/8} }{K_2 \, n^{1/8}}$ in the last display, we deduce from \eqref{eq:step31} and \eqref{eq:step32} that
\begin{align*}
\sup_{t\geq0}\bigg|\P(\hat{Q}_{\max}\leq t)-\P\bigg(\max_{\bv\in\mathbb{N}_{\epsilon_1}}|G_{\bv}|\leq t\bigg)\bigg|\lesssim K_2^2 \frac{(\log p_{\epsilon_1})^{9/8} }{n^{1/8}}+K_2^2 \frac{ (\log p_{\epsilon_1})^{9/2} }{n^{1/2}}.
\end{align*}
This completes the proof. 

\subsubsection{Proof of Lemma \ref{lem:obs}}
\begin{proof}
Noting that
\begin{align*}
\|\bv_{\bSigma}-\tilde\bv_{\bSigma}\|_{\bSigma}^2&=2-2\frac{\bv^{\T}\bSigma\tilde\bv}{(\bv^{\T}\bSigma\bv)^{1/2} (\tilde\bv^{\T}\bSigma\tilde\bv)^{1/2}}\\
&=\frac{(\bv-\tilde\bv)^{\T}\bSigma(\bv-\tilde\bv)- \{ (\bv^{\T}\bSigma\bv)^{1/2}-(\tilde\bv^{\T}\bSigma\tilde\bv)^{1/2} \}^2}{(\bv^{\T}\bSigma\bv)^{1/2}(\tilde\bv^{\T}\bSigma\tilde\bv)^{1/2}},
\end{align*}
we have
\begin{align}\label{eq:discret1}
\|\bv_{\bSigma}-\tilde\bv_{\bSigma}\|_{\bSigma}^2\leq\frac{(\bv-\tilde\bv)^{\T}\bSigma(\bv-\tilde\bv)}{(\bv^{\T}\bSigma\bv)^{1/2}(\tilde\bv^{\T}\bSigma\tilde\bv)^{1/2}}\leq\gamma_s^2(\bSigma)\cdot\|\bv-\tilde\bv\|_2^2.
\end{align}
By the triangle inequality,
\begin{align*}
\big||\bv_{\bSigma}^{\T}\Mb\bv_{\bSigma}|-|\tilde\bv_{\bSigma}^{\T}\Mb\tilde\bv_{\bSigma}|\big|&\leq\big|\bv_{\bSigma}^{\T}\Mb\bv_{\bSigma}-\tilde\bv_{\bSigma}^{\T}\Mb\tilde\bv_{\bSigma}\big|\notag\\
&=\big|\bv_{\bSigma}^{\T}\Mb(\bv_{\bSigma}-\tilde\bv_{\bSigma})+(\bv_{\bSigma}-\tilde\bv_{\bSigma})^{\T}\Mb\tilde\bv_{\bSigma}\big|.
\end{align*}
It follows that
\begin{align*}
\big||\bv_{\bSigma}^{\T}\Mb\bv_{\bSigma}|-|\tilde\bv_{\bSigma}^{\T}\Mb\tilde\bv_{\bSigma}|\big|\leq\big|\bv_{\bSigma}^{\T}\Mb(\bv_{\bSigma}-\tilde\bv_{\bSigma})\big|+\big|(\bv_{\bSigma}-\tilde\bv_{\bSigma})^{\T}\Mb\tilde\bv_{\bSigma}\big|.
\end{align*}
Using Lemma \ref{lem:sigma norm}, we deduce that
\begin{align}
\big||\bv_{\bSigma}^{\T}\Mb\bv_{\bSigma}|-|\tilde\bv_{\bSigma}^{\T}\Mb\tilde\bv_{\bSigma}|\big|\leq & \, \bigg|\bigg(\frac{\bv_{\bSigma}}{\|\bv_{\bSigma}\|_2}\bigg)^{\T}_{\bSigma}\Mb\bigg(\frac{\bv_{\bSigma}-\tilde\bv_{\bSigma}}{\|\bv_{\bSigma}-\tilde\bv_{\bSigma}\|_2}\bigg)_{\bSigma}\bigg|\cdot\|\bv_{\bSigma}-\tilde\bv_{\bSigma}\|_{\bSigma}\notag\\
&+\bigg|\bigg(\frac{\bv_{\bSigma}-\tilde\bv_{\bSigma}}{\|\bv_{\bSigma}-\tilde\bv_{\bSigma}\|_2}\bigg)_{\bSigma}^{\T}\Mb\bigg(\frac{\tilde\bv_{\bSigma}}{\|\tilde\bv_{\bSigma}\|_2}\bigg)_{\bSigma}\bigg|\cdot\|\bv_{\bSigma}-\tilde\bv_{\bSigma}\|_{\bSigma} \notag \\
 \leq & \, 2\|\bv_{\bSigma}-\tilde\bv_{\bSigma}\|_{\bSigma}\cdot\sup_{\bv\in\mathbb{V}(s,d)}|\bv_{\bSigma}^{\T}\Mb\bv_{\bSigma}|. \label{eq:discret2}
\end{align}
Combining \eqref{eq:discret1} and \eqref{eq:discret2} gives
\begin{align*}
\big||\bv_{\bSigma}^{\T}\Mb\bv_{\bSigma}|-|\tilde\bv_{\bSigma}^{\T}\Mb\tilde\bv_{\bSigma}|\big|\leq&2\gamma_s(\bSigma)\|\bv-\tilde\bv\|_2 \sup_{\bv\in\mathbb{V}(s,d)}|\bv_{\bSigma}^{\T}\Mb\bv_{\bSigma}|,
\end{align*}
as desired.
\end{proof}

\subsubsection{Proof of Theorem \ref{thm:multiplier_bootstrap}}

\begin{proof}
Based on the $\epsilon_1$-net $\mathbb{N}_{\epsilon_1}^0$ described in Lemma \ref{lem:discret} and the corresponding $p_1$-dimensional Gaussian random vector $(G_1,\ldots,$ $G_{p_1})^{\T}$ introduced in the proof of Theorem \ref{thm:limiting_dist} with $p_1:=p_{\epsilon_1}^0=|\mathbb{N}^0_{\epsilon_1}|$, we aim to show that
\begin{align*}
\sup_{t\geq0}\bigg|\P\bigg(\max_{1\leq j\leq p_1}|G_j|\leq t\bigg)-\P\bigg(\hat B_{\max} \leq t  \, \bigg| \bX_1,\ldots,\bX_n\bigg)\bigg|=o_{\mathbb{P}}(1).
\end{align*}
In view of Theorem \ref{thm:limiting_dist}, it suffices to prove that
\begin{align}\label{eq:boot}
\sup_{t\geq0}\bigg|\P\bigg(\max_{1\leq j\leq p_1}|G_j|\leq t\bigg)-\P\bigg(\max_{1\leq j\leq p_1}|\hat B_{ j}|\leq t \, \bigg| \bX_1,\ldots,\bX_n\bigg)\bigg|=o_{\mathbb{P}}(1),
\end{align}
where $\{ \hat B_j \}_{j=1}^{p_1}= \{ \hat B_{\bv} \}_{\bv\in\mathbb{N}_{\epsilon_1}^0}$. In particular, we note that 
\[
\sup_{t\geq 0}\bigg|\P\bigg(\hat B_{\max} \leq t  \, \bigg| \bX_1,\ldots,\bX_n\bigg)-\P\bigg(\max_{1\leq j\leq p_1}|\hat B_{ j}|\leq t \, \bigg|\bX_1,\ldots,\bX_n\bigg)\bigg|=O_{\mathbb{P}}\bigg\{\frac{ \gamma^{9/8}_n(s,d)}{n^{1/8}} \bigg\},
\]
via the proof of Theorem \ref{thm:limiting_dist}.

By Lemma \ref{lem:comparison}, we have, 
\begin{align}\label{eq:T2}
&\sup_{t\geq0}\bigg|\P\bigg(\max_{1\leq j\leq p_1}|G_j|\leq t\bigg)-\P\bigg(\max_{1\leq j\leq p_1}|\hat B_{ j}|\leq t \, \bigg| \bX_1,\ldots,\bX_n\bigg)\bigg|\notag\\
\lesssim &~ \Delta_G^{1/3} \{ \log (2p_1)\}^{1/3}\big\{ 1\vee 2\log(2p_1)\vee\log(1/\Delta_G)\big\}^{1/3}\notag\\
\lesssim&~ \Delta_G^{1/3} \{ \log (2p_1)\}^{2/3}\vee \Delta_G^{1/3} \{ \log(1/\Delta_G ) \}^{1/3} \{\log (2p_1)\}^{1/3},
\end{align}
where $p_1=|\mathbb{N}^0_{\epsilon_1}|$ satisfies $\log p_1\lesssim \gamma_n(s,d) = s\log ( \gamma_s \cdot ed /s  ) \vee s\log n$ and
\begin{align*}
\Delta_G:=\max_{1\leq j\leq k\leq p_1}\big| \E ( G_j G_k ) - \E( \hat B_{j} \hat B_{k} | \bX_1,\ldots,\bX_n ) \big|.
\end{align*}

Next we bound $\Delta_G$. For $1\leq j\leq k \leq p_1$, we have
\begin{align*}
\E ( G_j G_k )=\E\bigg\{ \frac{(\bv_j^{\T} \bX_i)^2(\bv_k^{\T} \bX_i)^2}{\|\bv_j\|_{\bSigma}^2\|\bv_k\|_{\bSigma}^2}\bigg\}-1.
\end{align*}
By definition \eqref{eq:MBdef}, we have, for $j=1,\ldots,p_1$,
\begin{align*}
\hat B_{ j}=\frac{1}{\sqrt{n}}\sum_{i=1}^n\xi_i\bigg\{ \frac{(\bv_j^{\T}\bX_i)^2}{\|\bv_j\|_{\bSigma}^2}-1\bigg\}.
\end{align*}
It follows that, for $1\leq j\leq k\leq p_1$,
\begin{align*}
&\E( \hat B_{j} \hat B_{k} | \bX_1,\ldots,\bX_n )\\
=&\frac{1}{n}\sum_{i=1}^n\bigg\{ \frac{(\bv_j^{\T} \bX_i)^2(\bv_k^{\T} \bX_i)^2}{\|\bv_j\|_{\bSigma}^2\|\bv_k \|_{\bSigma}^2}\bigg\} - \frac{1}{n}\sum_{i=1}^n\bigg\{ \frac{(\bv_j^{\T}\bX_i)^2}{\|\bv_j\|_{\bSigma}^2}-1\bigg\}-\frac{1}{n}\sum_{i=1}^n\bigg\{ \frac{(\bv_k^{\T}\bX_i)^2}{\|\bv_k \|_{\bSigma}^2}-1\bigg\} - 1.
\end{align*}
For simplicity, we define
\begin{align}\label{eq:def_yij}
W_{ij} =\frac{\bv_j^{\T}\bX_i}{\|\bv_j\|_{\bSigma}}~~~{\rm and}~~~W_{i k} =\frac{\bv_k^{\T}\bX_i}{\|\bv_k \|_{\bSigma}}, \ \  i=1,\ldots, n, \, 1\leq j\leq k\leq p_1.
\end{align}
In this notation, we have
\begin{align*}
\Delta_G=&\max_{1\leq j\leq k \leq p_1}\bigg|\frac{1}{n}\sum_{i=1}^n \{ (W_{ij}W_{ik})^2-\E (W_{ij}W_{ik})^2 \} +\frac{1}{n}\sum_{i=1}^n (W_{ij}^2-1)+\frac{1}{n}\sum_{i=1}^n (W_{ik}^2-1)\bigg|\notag\\
\leq&\max_{1\leq j\leq k \leq p_1}\bigg|\frac{1}{n}\sum_{i=1}^n    \{ (W_{ij}W_{ik})^2-\E (W_{ij}W_{ik})^2 \} \bigg|+2\max_{1\leq j\leq p_1}\bigg|\frac{\bv_j^{\T}(\hat\bSigma-\bSigma)\bv_j}{\bv_j^{\T}\bSigma\bv_j}\bigg|.
\end{align*}
Further, define
\begin{align*}
\Delta_{G,1} =\max_{1\leq j\leq k \leq p_1}\bigg|\frac{1}{n}\sum_{i=1}^n \{ (W_{ij}W_{i k})^2-\E (W_{ij}W_{i k})^2 \}  \bigg|, \ \ 
\Delta_{G,2}  =2\max_{1\leq j \leq p_1}\bigg|\frac{\bv_j^{\T}(\hat\bSigma-\bSigma)\bv_j}{\bv_j^{\T}\bSigma\bv_j}\bigg|.
\end{align*}
The following lemma gives an upper bound for $\Delta_{G,1}$.

\begin{lemma}\label{lem:LT2}
For any $M>0$, there exists an absolute positive constant $C_{21}$ only depending on $M$ such that
\begin{align*}
\P\left[ \max_{1\leq j, k \leq p_1}\bigg|\frac{1}{n}\sum_{i=1}^n  \{ (W_{ij}W_{i k})^2-\E (W_{ij}W_{i k})^2 \}  \bigg|\geq C_{21}\sqrt{\frac{\log p_1}{n}} \right]=O(p_1^{-M}),
\end{align*}
where $p_1=|\mathbb{N}^0_{\epsilon_1}|$, $W_{ij}$ and $W_{i k}$ for $1\leq i\leq n$ and $1\leq j, k \leq p_1$ are defined in \eqref{eq:def_yij}.
\end{lemma}

\medskip
By Lemma \ref{lem:LT2}, there exists an absolute positive constant $C_{21}$ depending only on $M$ such that
\begin{align}\label{eq:bound1}
\P\left(\Delta_{G,1}\geq C_{21}\sqrt{\frac{\log p_1}{n}} \right) =O(p_1^{-M}).
\end{align}
Turning to $\Delta_{G,2}$, by Lemma \ref{lem:connecting}, there exists a constant $C>0$ depending only on $K_2$ such that
\begin{align}\label{eq:bound2}
\P\bigg[\Delta_{G,2}\geq C \bigg\{ \frac{ \gamma^{1/2}_n(s,d)}{n^{1/2}}+\gamma_n(s,d) \frac{\log n}{n} \bigg\} \bigg]\leq \frac{4}{n}.
\end{align}
Combining \eqref{eq:bound1} and \eqref{eq:bound2}, we have with probability greater than $1-O(p_1^{-M})$,
\begin{align}\label{eq:T22}
 \Delta_{G}^{1/3} \{ \log(2p_1) \}^{2/3}\leq C  \frac{ \{ \log (2p_1) \}^{2/3}  \gamma^{1/6}_n(s,d) }{n^{1/6}} .
\end{align}
Since $x\mapsto x\log(1/x) $ is non-decreasing for $0< x \leq e^{-1}$, we have with probability greater than $1-O(p_1^{-M})$,
\begin{align}\label{eq:T23}
 \Delta_{G}^{1/3} \{ \log (2p_1) \}^{1/3} \{ \log(1 / \Delta_{G})\}^{1/3} \leq C    \frac{  \{ \log (2p_1) \}^{2/3} \gamma^{1/6}_n(s,d) }{n^{1/6}}  .
\end{align}
Putting \eqref{eq:T2}, \eqref{eq:T22}, and \eqref{eq:T23} together, we conclude that
\begin{align*}
\P\bigg\{ \sup_{t\geq0}\bigg|\P\bigg(\!\max_{1\leq j\leq p_1}\!|G_j|\!\leq\! t\bigg)\!-\!\P\bigg(\max_{1\leq j\leq p_1}|\hat B_j |\!\leq\! t \, \bigg| \bX_1,\ldots,\bX_n\bigg)\bigg|\!\geq\! C \frac{ \gamma_n^{5/6}(s,d) }{n^{1/6}}  \bigg\}   \!=\! O(p_1^{-M}).
\end{align*}
This proves \eqref{eq:boot}.

Finally, using Theorem \ref{thm:limiting_dist}, we deduce that for any $M>0$, there exists a constant $C_M>0$ depending only on $M$ and $K_1$ such that 
\begin{align*}
\P\bigg[ \sup_{t\geq0}\big|\P\big(\hat{Q}_{\max}\!\leq\! t\big)\!-\!\P (\hat B_{\max} \!\leq\! t  \, |  \bX_1,\ldots,\bX_n ) \big|\!\geq\! C_M\bigg\{ \frac{ \gamma^{9/8}_n(s,d)}{n^{1/8}}+\frac{ \gamma_n^{5/6}(s,d) }{n^{1/6}}  \bigg\} \bigg]
 \lesssim p_1^{-M}.
\end{align*}
This completes the proof. 
\end{proof}

\subsubsection{Proof of Theorems \ref{thm:limiting_dist2} and \ref{thm:limiting_dist3}}

Theorems \ref{thm:limiting_dist2} and \ref{thm:limiting_dist3} can be proved based on similar arguments used in the proofs of Theorems \ref{thm:limiting_dist} and Theorem \ref{thm:multiplier_bootstrap}. The details are hence omitted. 

\subsubsection{Proof of Theorem \ref{thm:testing}}\label{sec:testing}
\begin{proof}
To begin with, we introduce the following notations. Define
\begin{align*}
\cQ_{ \bv} =\sqrt{\frac{nm}{n+m}} \frac{\bv^{\T}(\hat\bSigma_1-\hat\bSigma_2)\bv}{\bv^{\T}\bSigma_2\bv}, \ \ \cQ_{\max}  =\sup_{\bv\in\mathbb{V}(s,d)}|\cQ_{\bv}|, \ \ \mbox{ and } \ \  \cM_{\max}  =\max_{\bv\in\mathbb{N}^0_{\epsilon_3}}|\cQ_{ \bv}|,
\end{align*}
where $\epsilon_3:=\{ m\gamma_s(\bSigma_2)\}^{-1}$. 

We divide the proof into three main steps. (i) First, using the discretized version $\cM_{\max}$ as a bridge, we show that $\cQ_{\max}$ converges weakly to the maximum of a Gaussian sequence. (ii) Next we show that the difference between $\cQ_{\max}$ and the test statistic $\hat\cQ_{\max}$ is negligible asymptotically. (iii) Finally, we show that the Gaussian maximum can be approximated by its multiplier bootstrap counterpart. The technical details are stated as lemmas with their proofs deferred to Section \ref{sec:Main_Lemmas}.

\begin{lemma}\label{lem:testinglem1}
Let Assumptions \ref{cdt5} and \ref{cdt6} be satisfied. Under the null hypothesis ${\bf H_0}: \bSigma_1=\bSigma_2$, we have the following two assertions hold. 

(i) We have
\begin{align}\label{eq:cqg1}
\P\bigg[\cQ_{\max} \leq C_{31}L_2 \bigg\{   \gamma_{m}^{1/2}(s,d)  +  (\log m)^{1/2} + \gamma_{m}(s,d) \frac{\log m}{\sqrt{m}}\bigg\} \bigg] \geq1-\frac{4}{n}-\frac{4}{m},
\end{align}
where $C_{31} >0$ is an absolute constant, $\gamma_{m}(s,d):=s\log\{ \gamma_s(\bSigma_2)\frac{ed}{s}\} \vee s\log m$, and $L_2 :=L_1^2+1$. 

(ii) Let $\mathbb{N}^0_{\epsilon_3}=\{\bu_j\}_{j=1}^{p_3}$ be an $\epsilon_3$-net with $\epsilon_3=\{ m\gamma_s(\bSigma_2)\}^{-1}$ and $p_3 =|\mathbb{N}^0_{\epsilon_3}|$. Then, there exists a $p_3$-dimensional Gaussian random vector $(\cG_1,\ldots,\cG_{p_3})^{\T}$ satisfying
\begin{align}\label{eq:cGreq}
\E( \cG_j\cG_k )=\frac{n}{m(n+m)} \sum_{i=1}^{m} \E(  \cR_{ij}\cR_{ik}) , \ \ 1\leq j\leq k\leq p_3,
\end{align}
with
\begin{align}\label{eq:Rijdef}
\cR_{ij}:=\left\{\begin{array}{ll}
\frac{m}{n} \frac{\bu_j^{\T}(\bX_i\bX_i^{\T}-\bSigma_2)\bu_j}{\bu_j^{\T}\bSigma_2\bu_j}-\frac{\bu_j^{\T}(\bY_i\bY_i^{\T}-\bSigma_2)\bu_j}{\bu_j^{\T}\bSigma_2\bu_j}, & {\rm if }~ 1\leq i\leq n,\\
-\frac{\bu_j^{\T}(\bY_{i} \bY_{i}^{\T}-\bSigma_2)\bu_j}{\bu_j^{\T}\bSigma_2\bu_j}, & {\rm if }~ n+1\leq i\leq m,
\end{array}\right.
\end{align}
(here, without loss of generality, we assume $n\leq m$) such that
\begin{align}\label{eq:cqg2}
\P\bigg[\bigg|\cQ_{\max}-\max_{1\leq j\leq p_3 }|\cG_{j}| \bigg|\geq C_{32} \frac{K_0^4 \gamma^{5/8}_{m}(s,d) }{L_2 \, m^{1/8}} \bigg] \lesssim L_2^2 \frac{  \gamma_{m}^{1/8}(s,d)}{m^{1/8}}+L_2^2 \frac{ \gamma^{9/2}_{m}(s,d)}{m^{1/2}},
\end{align}
where $C_{32}>0$ is an absolute constant and $K_0$ is the constant in Lemma \ref{lem:maximal_ineq} by taking $\alpha=1$.
\end{lemma}

\begin{lemma}\label{lem:testinglem2}
Let Assumptions \ref{cdt5} and \ref{cdt6} be satisfied. Under the null hypothesis ${\bf H_0}: \bSigma_1=\bSigma_2$, we have, as $n, m \to \infty$,
\begin{align}\label{eq:cqq1}
\P\left[|\hat\cQ_{\max} - \cQ_{\max} |\leq C_{33} L_2 \sqrt{\frac{nm}{n+m}} \bigg\{  \frac{ \gamma^{1/2}_{m}(s,d)}{\sqrt{m}} + \gamma_{m}(s,d)\frac{\log m}{m}\bigg\}^2\right]\notag \\   
 \geq1-\frac{4}{n}-\frac{4}{m},
\end{align}
where $C_{33}>0$ is an absolute constant.
\end{lemma}

\begin{lemma}\label{lem:testinglem3}
Let Assumptions \ref{cdt5} and \ref{cdt6} be satisfied. Under the null hypothesis ${\bf H_0}: \bSigma_1=\bSigma_2$, we have, as $n, m \to \infty$,
\begin{align*}
\sup_{t\geq0}\bigg|\P\bigg( \max_{1\leq j\leq p_3 }|\cG_{j}| \leq t\bigg)-\P(\hat\cB_{\max}\leq t \, |  \bX_1,\ldots,\bX_{n},\bY_1,\ldots,\bY_{m})\bigg|=o_{\mathbb{P}}(1).
\end{align*}
\end{lemma}

\medskip
Combining \eqref{eq:cqg2} and \eqref{eq:cqq1} we deduce that there exists an absolute constant $C>0$ such that
\begin{align*}
\P\bigg[ \bigg|\hat\cQ_{\max} - \max_{1\leq j\leq p_3 }|\cG_{j}|  \bigg|\geq C \frac{K_0^4  \gamma^{5/8}_{m}(s,d)}{L_2 \, m^{1/8}} \bigg] \lesssim L_2^2 \frac{ \gamma^{1/8}_{m}(s,d) }{m^{1/8}}+L_2^2 \frac{  \gamma^{9/2}_{m}(s,d) }{m^{1/2}}.
\end{align*}
Using arguments similar to those used in the proof of Theorem \ref{thm:limiting_dist}, we deduce that
\begin{align*}
\sup_{t\geq0}\bigg|\P(\hat\cQ_{\max} \leq t)-\P\bigg( \max_{1\leq j\leq p_3 }|\cG_{j}| \leq t\bigg)\bigg|\lesssim L_2^2 \frac{  \gamma^{9/8}_{m}(s,d)}{m^{1/8}}+L_2^2 \frac{ \gamma^{9/2}_{m}(s,d)}{m^{1/2}}.
\end{align*}
This, together with Lemma \ref{lem:testinglem3} yields that
\begin{align*}
\sup_{t\geq 0}\big| \P( \hat\cQ_{\max} \leq t) - \P(\hat\cB_{\max} \leq t  \, |  \bX_1,\ldots,\bX_{n},\bY_1,\ldots,\bY_{m})  \big| = o_{\mathbb{P}}(1),
\end{align*}
which completes the proof. 
\end{proof}

\subsubsection{Proof of Theorem \ref{thm:power}}
\begin{proof}
It is equivalent to proving that for $\lambda>0$ sufficiently large,
\begin{align}\label{eq:poweraim}
\inf_{(\bSigma_1,\bSigma_2)\in\mathbb{M}(\lambda)}\P_{(\bSigma_1,\bSigma_2)}(\hat\cQ_{\max}\geq q_{\alpha})=1-o(1).
\end{align}
First we claim that $q_{\alpha}=O_{\mathbb{P}} \{ \sqrt{s\log (ed/s)} \}$. To see this, it suffices to show that
\begin{align*}
\hat\cB_{\max} = O_{\mathbb{P}} \{ \sqrt{s\log (ed/s)}  \} .
\end{align*}
It suffices to show
\begin{align*}
& \sup_{\bv\in\mathbb{V}(s,d)}\bigg|\sqrt{\frac{n+m}{nm}} \frac{\bv^\T\{ \sum_{i=1}^{n}\xi_i(\bX_i\bX_i^\T-\hat\bSigma_1) / n \!-\!  \sum_{i=1}^{m}\eta_i(\bY_i\bY_i^\T-\hat\bSigma_2) / m \}\bv}{\bv^\T(\bSigma_1/n+\bSigma_2/m)\bv}\bigg|  \\ 
 =&O_{\mathbb{P}} \{ \sqrt{s\log (ed/s)} \},
\end{align*}
since, by exactly the same argument as in the proof of Lemma \ref{lem:testinglem2}, the difference between 
\[
\sup_{\bv\in\mathbb{V}(s,d)}\bigg|\sqrt{\frac{n+m}{nm}} \frac{\bv^\T\{ \sum_{i=1}^{n}\xi_i(\bX_i\bX_i^\T-\hat\bSigma_1) / n \!-\!  \sum_{i=1}^{m}\eta_i(\bY_i\bY_i^\T-\hat\bSigma_2) / m \}\bv}{\bv^\T(\bSigma_1/n+\bSigma_2/m)\bv}\bigg|
\]
and
\[
\sup_{\bv\in\mathbb{V}(s,d)}\bigg|\sqrt{\frac{n+m}{nm}} \frac{\bv^\T\{ \sum_{i=1}^{n}\xi_i(\bX_i\bX_i^\T-\hat\bSigma_1) / n \!-\!  \sum_{i=1}^{m}\eta_i(\bY_i\bY_i^\T-\hat\bSigma_2) / m \}\bv}{\bv^\T(\hat\bSigma_1/n+\hat\bSigma_2/m)\bv}\bigg|
\]
is of order $O_{\mathbb{P}}\{ \sqrt{s\log (ed/s)}  \}$. It then reduces to show
\begin{align*}
\sup_{\bv\in\mathbb{V}(s,d)}\bigg|\bv^{\T}\bigg\{ \frac{1}{n}\sum_{i=1}^{n}\xi_i(\bX_i\bX_i^{\T}-\hat\bSigma_1)-\frac{1}{m}\sum_{i=1}^{m}\eta_i(\bY_i\bY_i^{\T}-\hat\bSigma_2)\bigg\} \bv\bigg|=O_{\mathbb{P}}\left\{ \sqrt{\frac{s\log (ed/s)}{m}} \right\},
\end{align*}
since we have, for any $\bv\in\mathbb{V}(s,d)$,
\begin{align*}
\bigg|\frac{\bv^\T\{ \sum_{i=1}^{n}\xi_i(\bX_i\bX_i^\T-\hat\bSigma_1) / n \!-\!  \sum_{i=1}^{m}\eta_i(\bY_i\bY_i^\T-\hat\bSigma_2) / m \}\bv}{\bv^\T(\bSigma_1/n+\bSigma_2/m)\bv}\bigg| \\
\lesssim  n\cdot\bigg|\bv^{\T}\bigg\{ \frac{1}{n}\sum_{i=1}^{n}\xi_i(\bX_i\bX_i^{\T}-\hat\bSigma_1)-\frac{1}{m}\sum_{i=1}^{m}\eta_i(\bY_i\bY_i^{\T}-\hat\bSigma_2)\bigg\} \bv\bigg|.
\end{align*}
This is due to the fact that $\bSigma_1,\bSigma_2\in \mathbb{M}(\lambda)$ and $m\asymp n$. Then, we can further write
\begin{align*}
&\P\left[  \sup_{\bv\in\mathbb{V}(s,d)}\bigg|\bv^{\T}\bigg\{ \frac{1}{n}\sum_{i=1}^{n}\xi_i(\bX_i\bX_i^{\T}\!-\!\hat\bSigma_1)\!-\!\frac{1}{m}\sum_{i=1}^{m}\eta_i(\bY_i\bY_i^{\T}\!-\!\hat\bSigma_2)\bigg\} \bv\bigg| \!\geq\! C_{41}\sqrt{\frac{s\log (ed/s)}{m}} \,\right] \notag\\
\leq&~ \P\left\{ H_1\geq\frac{C_{41}}{4} \sqrt{\frac{s\log (ed/s)}{m}}\right\} +\P\left\{ H_2\geq\frac{C_{41}}{4} \sqrt{\frac{s\log (ed/s)}{m}}\right\} \notag\\
&~+\P\left\{ H_3\geq\frac{C_{41}}{4} \sqrt{\frac{s\log (ed/s)}{m}}\right\} + \P\left\{ H_4\geq\frac{C_{41}}{4} \sqrt{\frac{s\log (ed/s)}{m}}\right\},
\end{align*}
where
\begin{align*}
H_1:=\sup_{\bv\in\mathbb{V}(s,d)} \bigg|\frac{1}{n}\sum_{i=1}^{n}\xi_i \bv^{\T}(\bX_i\bX_i^{\T}-\bSigma_1)\bv\bigg|,~~~H_2:=\sup_{\bv\in\mathbb{V}(s,d)}\bigg|\frac{1}{m}\sum_{i=1}^{m}\eta_i \bv^{\T}(\bY_i\bY_i^{\T}-\bSigma_2)\bv\bigg|,\notag\\
H_3:=\sup_{\bv\in\mathbb{V}(s,d)} \bigg|\frac{1}{n}\sum_{i=1}^{n}\xi_i \bv^{\T}(\hat\bSigma_1-\bSigma_1)\bv \bigg|, ~~~{\rm and}~~~H_4:=\sup_{\bv\in\mathbb{V}(s,d)} \bigg|\frac{1}{m}\sum_{i=1}^{m} \eta_i \bv^{\T}(\hat\bSigma_2-\bSigma_2)\bv \bigg|.
\end{align*}
We bound $H_1,H_2, H_3$, and $H_4$ respectively. Without loss of generality, we only need to consider $H_1$ and $H_3$. For $H_1$, define $\bar\xi_i  = \xi_i I(|\xi_i|\leq\tau\sqrt{\log n})$ for some sufficiently large $\tau >0$. Using the standard $\epsilon$-net argument, it can be shown that (using Lemma 5.4 in \cite{vershynin2010introduction})
\begin{align*}
& \P\left\{ H_1\geq\frac{C_{41}}{4} \sqrt{\frac{s\log (ed/s)}{m}}\right\} \\
\leq&\binom{d}{s}9^s\cdot\P\left\{ \bigg| \frac{1}{n}\sum_{i=1}^{n}\xi_i \bv^{\T}(\bX_i\bX_i^{\T}-\bSigma_1)\bv \bigg|  \geq\frac{C_{41}}{8} \sqrt{\frac{s\log (ed/s)}{m}} \right\}
\end{align*}
and
\begin{align*}
& \P\left\{ \bigg|\frac{1}{n}\sum_{i=1}^{n}\xi_i \bv^{\T}(\bX_i\bX_i^{\T}-\bSigma_1)\bv\bigg|\geq\frac{C_{41}}{8} \sqrt{\frac{s\log (ed/s)}{m}}\right\} \\
  \leq& n\max_{1\leq i\leq n}\P\big( |\xi_i|>\tau\sqrt{\log n} \, \big) + \P\left\{ \bigg|\frac{1}{n}\sum_{i=1}^{n}\bar\xi_i \bv^{\T}(\bX_i\bX_i^{\T}-\bSigma_1)\bv\bigg|\geq\frac{C_{41}}{8} \sqrt{\frac{s\log (ed/s)}{m}}\right\}.
\end{align*}
Similar to Lemma \ref{lem:LT2}, define $V_i:=\bar\xi_i \bv^{\T}(\bX_i\bX_i^{\T}-\bSigma_1)\bv$ and, by Markov's inequality, we have for any $t>0$,
\begin{align*}
& \P\left\{ \frac{1}{n}\sum_{i=1}^{n}\bar\xi_i \bv^{\T}(\bX_i\bX_i^{\T}-\bSigma_1)\bv\geq\frac{C_{41}}{8} \sqrt{\frac{s\log (ed/s)}{m}}\right\} \\ 
 \leq& \exp\bigg\{ -\frac{C_{41}}{8}t\sqrt{ms\log(ed/s)}\bigg\} \prod_{i=1}^n\E\exp(tV_i).
\end{align*}
Taking $t=\sqrt{s\log(ed/s)/m}$, it follows
\begin{align*}
&\P\left\{ \frac{1}{n}\sum_{i=1}^{n}\bar\xi_i \bv^{\T}(\bX_i\bX_i^{\T}-\bSigma_1)\bv\geq\frac{C_{41}}{8} \sqrt{\frac{s\log (ed/s)}{m}}\right\}\\
\leq&\exp\bigg( -\frac{C_{41}}{8}s\log(ed/s)+\frac{s\log(ed/s)}{m}\sum_{i=1}^n\E\bigg[V_i^2\exp\bigg\{ \sqrt{\frac{s\log(ed/s)}{m}}|V_i| \bigg\} \bigg] \bigg).
\end{align*}
Similar to \eqref{eq:LT3}, we get $H_1=O_{\mathbb{P}} \{ \sqrt{s\log(ed/s)/m}  \}$ as long as $s\log (ed/s)\log n=o(n)$. Furthermore, using the fact
\begin{align*}
H_3\leq\sup_{\bv\in\mathbb{V}(s,d)}|\bv^{\T}(\hat\bSigma_1-\bSigma_1)\bv|\cdot \frac{1}{n}\sum_{i=1}^{n}|\xi_i| =O_{\mathbb{P}} \{  \sqrt{s\log(ed/s)/m} \},
\end{align*}
we deduce that $H_3=O_{\mathbb{P}} \{ \sqrt{s\log(ed/s)/m} \}$. Putting together the pieces, we conclude that $q_{\alpha}=O_{\mathbb{P}} \{ \sqrt{s\log (ed/s)} \}$.

Secondly, we study $\hat\cQ_{\max}$. As in Lemma \ref{lem:testinglem2}, we bound $\cQ_{\max}'$ instead, where
\begin{align*}
\cQ_{\max}':= \sqrt{\frac{n+m}{nm}} \frac{\bv^\T(\hat\bSigma_1-\hat\bSigma_2)\bv}{\bv^\T(\bSigma_1/n+\bSigma_2/m)\bv}.
\end{align*}
This is, again, because the difference between them is of order $O_{\mathbb{P}}\{ \sqrt{s\log(ed/s)} \}$. Note that
\begin{align*}
\cQ_{\max}'\geq&\sqrt{\frac{n+m}{nm}} \sup_{\bv\in\mathbb{V}(s,d)}\bigg|\frac{\bv^\T(\bSigma_1-\bSigma_2)\bv}{\bv^\T(\bSigma_1/n+\bSigma_2/m)\bv}\bigg|-(H_5+H_6),
\end{align*}
where
\begin{align*}
H_5 & := \sqrt{\frac{n+m}{nm}} \sup_{\bv\in\mathbb{V}(s,d)}\bigg|\frac{\bv^\T(\hat\bSigma_1-\bSigma_1)\bv}{\bv^\T(\bSigma_1/n\!+\!\bSigma_2/m)\bv}\bigg|,  \\
H_6 & := \sqrt{\frac{n+m}{nm}} \sup_{\bv\in\mathbb{V}(s,d)}\bigg|\frac{\bv^\T(\hat\bSigma_2-\bSigma_2)\bv}{\bv^\T(\bSigma_1/n\!+\!\bSigma_2/m)\bv}\bigg|.
\end{align*}
Equation \eqref{eq:poweraim} then follows from the fact that $H_5 + H_6 =O_{\mathbb{P}} \{ \sqrt{s\log (ed/s)}  \}$. This completes the proof.
\end{proof}

\subsubsection{Proof of Theorem \ref{thm:lower_bound}}
\begin{proof}
Define the class of rank one perturbations of the identity matrix as follows:
\begin{align*}
\mathbb{H}(\lambda):=\bigg\{\Mb=\Ib_d+\lambda\sqrt{\frac{s\log (ed/s)}{n}}\bv\bv^{\T}:\bv\in\mathbb{V}(s,d) \bigg\}.
\end{align*}
Then, it suffices to prove the conclusion with $\mathbb{M}(\lambda)$ replaced by all $\bSigma_1\in \mathbb{H}(\lambda)$ and $\bSigma_2=\Ib_d$. Let $\lambda$ be sufficiently small. For any two distributions $F$ and $G$, we write $F\otimes G$ to represent the product measure of $F$ and $G$. In particular, we use $F^{\otimes n}$ to denote the product distribution of $n$ independent copies of $F$. 
Recall that the minimax risk is lower bounded by the Bayesian risk. Define $\P^0_{\mu_{\lambda}} =\E \{ N_d(\zero,\Ib_d+\lambda\sqrt{s\log(ed/s)/n}\bv\bv^{\T})^{\otimes n}\otimes N_d(\zero,\Ib_d)^{\otimes m} \}$ to be the mixture alternative distribution with a prior distribution on $\bv$ with $\bv$ taking values uniformly in $\mathbb{V}(s,d)$: 
\[
\P^0_{\mu_{\lambda}}(A):=\int_A{\rm d}N_d\bigg(\zero,\Ib_d+\lambda\sqrt{\frac{s\log(ed/s)}{n}\bs\bs^\T}\bigg)^{\otimes n}\cdot {\rm d}N_d(\zero,\Ib_d)^{\otimes m}{\rm d}{\cS}(\bs),
\]
where $A\in \reals^{(n+m)d}$ and $\cS$ denotes the uniform measure on $\mathbb{V}(s,d)$ with respect to the Haar measure.

Define $\P^0_{(\bSigma_1,\Ib_d)}$ to be the probability measure of $N_d(\zero,\bSigma_1)^{\otimes n}\otimes N_d(\zero,\Ib_d)^{\otimes m}$. In particular, let $\P^0_{(\Ib_d,\Ib_d)}$ be the probability measure of $N_d(\zero,\Ib_d)^{\otimes {(n+m)}}$. 
Note that, for any measurable set $A\subset\reals^{(n+m)d}$, the measure $\P^0_{\mu_{\lambda}}$ satisfies
\begin{align*}
&\sup_{\bSigma_1\in\mathbb{H}(\lambda)}\P^0_{(\bSigma_1,\Ib_d)}(A^c)\geq\P^0_{\mu_{\lambda}}(A^c).
\end{align*}
Also by the definition of the probability measure, we have
\begin{align*}
&1=\P^0_{\mu_{\lambda}}(A)+\P^0_{\mu_{\lambda}}(A^c).
\end{align*}
Due to the triangular inequality, we have
\begin{align*}
&\P^0_{\mu_{\lambda}}(A)\leq \P^0_{(\Ib_d,\Ib_d)}(A)+|\P^0_{\mu_{\lambda}}(A)-\P^0_{(\Ib_d,\Ib_d)}(A)|.
\end{align*}
Putting $A=\{\Phi_{\alpha}=1\}$, we deduce that
\begin{align}\label{eq:LB1}
\inf_{ \Phi_{\alpha}}\sup_{\bSigma_1\in\mathbb{H}(\lambda)}\P^0_{(\bSigma_1,\Ib_d)}( \Phi_{\alpha}=0)&\geq1-\alpha-\sup_{A:\P^0_{(\Ib_d,\Ib_d)}(A)\leq\alpha}\big|\P^0_{\mu_{\lambda}}(A)-\P^0_{(\Ib_d,\Ib_d)}(A)\big|\notag\\
&\geq1-\alpha-\frac{1}{2}\|\P^0_{\mu_{\lambda}}-\P^0_{(\Ib_d,\Ib_d)}\|_{\rm TV},
\end{align}
where $\|\P^0_{\mu_{\lambda}}-\P^0_{(\Ib_d,\Ib_d)}\|_{\rm TV}$ denotes the total variation distance between the two probability measures $\P^0_{\mu_{\lambda}}$ and $\P^0_{(\Ib_d,\Ib_d)}$. 

To finish the proof, we introduce another distance measurement over distributions. Let the $\chi^2$-divergence between two probability measures $\P_1$ and $\P_2$ be defined as
\begin{align*}
\chi^2(\P_1||\P_2) =\int\bigg(\frac{d\P_1}{d\P_2}-1\bigg)^2{\rm d}\P_2.
\end{align*}
In view of the proof of Proposition 2 in \cite{cai2013optimal}, there exists a function $g:(0,1/36)\mapsto(1,\infty)$ with $g(0+)=1$ such that
\begin{align*}
\chi^2(\P^0_{\mu_{\lambda}}||\P^0_{(\Ib_d,\Ib_d)})\leq g(\beta_0)-1,
\end{align*}
where $\beta_0$ tends to zero as $\lambda \to 0$. Using the Pinsker's inequality (see, for example, Lemma 2.5 in \cite{tsybakov2008introduction})
\begin{align*}
\chi^2(\P^0_{\mu_{\lambda}}||\P^0_{(\Ib_d,\Ib_d)})\geq 2\, \|\P^0_{\mu_{\lambda}}-\P^0_{(\Ib_d,\Ib_d)}\|_{\rm TV}^2,
\end{align*}
we deduce from \eqref{eq:LB1} that
\begin{align*}
\inf_{ \Phi_{\alpha}}\sup_{\bSigma_1\in\mathbb{H}(\lambda)}\P^0_{(\bSigma_1,\Ib_d)}( \Phi_{\alpha}=0)\geq 1-\alpha-o(1).
\end{align*}
This completes the proof. 
\end{proof}

\subsection{Proofs of  the supporting lemmas}\label{sec:Main_Lemmas}
\subsubsection{Proof of Lemma \ref{lem:discret}}
\begin{proof}
For any $\epsilon\in(0,1)$ fixed and $\mathbb{I}\subseteq[d]$ subject to $|\mathbb{I}|=s$, let $\mathbb{S}^{s-1}_{\mathbb{I}}\subseteq\reals^d$ be the unit Euclidean sphere whose support is $\mathbb{I}$. Further, let $\mathbb{N}^0_{\mathbb{I},\epsilon}$ denote an $\epsilon$-net of $\mathbb{S}^{s-1}_{\mathbb{I}}$ with respect to the Euclidean metric $\rho_E$ satisfying $|\mathbb{N}^0_{\mathbb{I},\epsilon}|\leq (1+2/\epsilon)^s$. Due to the decomposition
\begin{align*}
\mathbb{V}(s,d)=\{\bv\in\mathbb{S}^{d-1}:|\bv|_0=s\}=\bigcup_{\mathbb{I}\subseteq[d]:|\mathbb{I}|=s}\{\bv\in\mathbb{S}^{d-1}:\supp(\bv)=\mathbb{I}\}=\bigcup_{\mathbb{I}\subseteq[d]:|\mathbb{I}|=s}\mathbb{S}_{\mathbb{I}}^{s-1},
\end{align*}
we can construct an $\epsilon$-net of $(\mathbb{V}(s,d),\rho_E)$ by $\mathbb{N}^0_\epsilon:=\bigcup_{\mathbb{I}\subseteq[d]:|\mathbb{I}|=s}\mathbb{N}^0_{\mathbb{I},\epsilon}$. Then, it is straightforward to see that 
\[
p^0_{\epsilon}=|\mathbb{N}^0_\epsilon| \leq \binom{d}{s}\bigg(1+\frac{2}{\epsilon}\bigg)^s. 
\]
Using the binomial coefficient bound 
\[
\binom{d}{s}\leq\bigg(\frac{ed}{s}\bigg)^s, 
\]
we get
\begin{align*}
\log p^0_{\epsilon}\lesssim s\log \frac{ed}{\epsilon s}.
\end{align*}

Next we prove the second assertion. For every $\bv\in\mathbb{V}(s,d)$ with support $\mathbb{I}$ and its $\epsilon$-net $\mathbb{N}_{\mathbb{I},\epsilon}$, we can find some $\tilde\bv\in\mathbb{N}_{\mathbb{I},\epsilon}$ satisfying that $\supp(\bv)=\supp(\tilde\bv)$ and $\|\bv-\tilde\bv\|_2\leq\epsilon$. By Lemma \ref{lem:obs}, we have
\begin{align*}
\big||\bv_{\bSigma}^{\T}(\hat\bSigma-\bSigma)\bv_{\bSigma}|-|\tilde\bv_{\bSigma}^{\T}(\hat\bSigma-\bSigma)\tilde\bv_{\bSigma}|\big|\leq& \, 2\gamma_s\|\bv-\tilde\bv\|_2\cdot\sup_{\bv\in\mathbb{V}(s,d)}|\bv_{\bSigma}^{\T}(\hat\bSigma-\bSigma)\bv_{\bSigma}|\notag\\
\leq& \, 2\gamma_s\epsilon\cdot\sup_{\bv\in\mathbb{V}(s,d)}|\bv_{\bSigma}^{\T}(\hat\bSigma-\bSigma)\bv_{\bSigma}|.
\end{align*}
Therefore, we have
\begin{align*}
\sup_{\bv\in\mathbb{V}(s,d):\supp(\bv)=\mathbb{I}}|\bv_{\bSigma}^{\T}(\hat\bSigma-\bSigma)\bv_{\bSigma}|\leq2\gamma_s\epsilon\cdot\sup_{\bv\in\mathbb{V}(s,d)}|\bv_{\bSigma}^{\T}(\hat\bSigma-\bSigma)\bv_{\bSigma}|+\max_{\bv\in\mathbb{N}_{\mathbb{I},\epsilon}}|\bv_{\bSigma}^{\T}(\hat\bSigma-\bSigma)\bv_{\bSigma}|.
\end{align*}
Taking maximum over $\mathbb{I}\subseteq[d]$ with $|\mathbb{I}|=s$ on both sides yields
\begin{align*}
\sup_{\bv\in\mathbb{V}(s,d)}|\bv_{\bSigma}^{\T}(\hat\bSigma-\bSigma)\bv_{\bSigma}|\leq2\gamma_s\epsilon\cdot\sup_{\bv\in\mathbb{V}(s,d)}|\bv_{\bSigma}^{\T}(\hat\bSigma-\bSigma)\bv_{\bSigma}|+\max_{\bv\in\mathbb{N}_{\epsilon}}|\bv_{\bSigma}^{\T}(\hat\bSigma-\bSigma)\bv_{\bSigma}|.
\end{align*}
Together, the last two displays imply $\hat{Q}_{\max} \leq 2\gamma_s\epsilon\cdot \hat{Q}_{\max} +M_{\max,\epsilon}$. 
This completes the proof.
\end{proof}

\subsubsection{Proof of Lemma \ref{lem:connecting}}
\begin{proof}
We follow a standard procedure. First we show concentration of $\hat{Q}_{\max}$ around its expectation $\E  \hat{Q}_{\max}$. Next we upper bound $\E \hat{Q}_{\max}$. To prove the concentration, we define for every $\bv\in\mathbb{V}(s,d)$ that
\begin{align*}
g_{\bv}(\bX_i) =\frac{\bv^{\T}\bX_i\bX_i^{\T}\bv}{\bv^{\T}\bSigma\bv}-1.
\end{align*}
By Lemma \ref{lem:tail_ineq}, there exists an absolute constant $C_{11}>0$ such that for every $t>0$,
\begin{align}\label{eq:lem32}
\P\bigg[ \hat{Q}_{\max}\leq2\E \hat{Q}_{\max} +\max\bigg\{\underbrace{2\sigma_{\bv}\sqrt{\frac{t}{n}} }_{J_1}, \underbrace{C_{11}\frac{t}{\sqrt{n}} \bigg\| \max_{1\leq i\leq n}\sup_{\bv\in\mathbb{V}(s,d)}|g_{\bv}(\bX_i)| \bigg\|_{\psi_1}}_{J_2}\bigg\}\bigg] \notag\\
 \geq 1 - 4e^{-t},
\end{align}
where $\sigma_{\bv}^2:=\sup_{\bv\in\mathbb{V}(s,d)}\sum_{i=1}^n\E g^2 _{\bv}(\bX_i)$. We first bound $J_1$ and $J_2$, starting with $J_1$. Under Assumption~\ref{cdt1}, we have
\begin{align}\label{eq:lem33}
\sigma_{\bv}^2\leq n\cdot \!\!\!\sup_{\bv\in\mathbb{V}(s,d)}\E\bigg\{\bigg(\frac{\bv^{\T}\bX_i\bX_i^{\T}\bv}{\bv^{\T}\bSigma\bv}-1\bigg)^2\bigg\} \leq 2 n \sup_{\bv\in\mathbb{S}^{d-1}}\|\bv^{\T}\bU_i\bU_i^{\T}\bv\!-\!1\|_{\psi_1}^2  \leq 2K_2^2 \, n,
\end{align}
and hence $\sigma_{\bv}\leq K_2\sqrt{2n}$. For $J_2$, using a similar argument as in the proof of Lemma \ref{lem:discret}, we deduce that for every $0< \epsilon < ( 2\gamma_s )^{-1}$,
\begin{align*}
\sup_{\bv\in\mathbb{V}(s,d)}\bigg|\frac{\bv^{\T}\bX_i\bX_i^{\T}\bv}{\bv^{\T}\bSigma\bv}-1\bigg|\leq ( 1-2\gamma_s\epsilon )^{-1}\max_{\bv\in\mathbb{N}^0_{\epsilon}}\bigg|\frac{\bv^{\T}\bX_i\bX_i^{\T}\bv}{\bv^{\T}\bSigma\bv}-1\bigg|.
\end{align*}
By taking $\epsilon=\epsilon_4:= ( 4\gamma_s)^{-1}$, we have
\begin{align*}
\bigg\|\max_{1\leq i\leq n}\sup_{\bv\in\mathbb{V}(s,d)}|g_{\bv}(\bX_i)| \bigg\|_{\psi_1}&=\bigg\|\max_{1\leq i\leq n}\sup_{\bv\in\mathbb{V}(s,d)}\bigg|\frac{\bv^{\T}\bX_i\bX_i^{\T}\bv}{\bv^{\T}\bSigma\bv}-1\bigg|\bigg\|_{\psi_1}\notag\\
&\lesssim\bigg\|\max_{1\leq i\leq n}\max_{\bv\in\mathbb{N}^0_{\epsilon_4}}\bigg|\frac{\bv^{\T}\bX_i\bX_i^{\T}\bv}{\bv^{\T}\bSigma\bv}-1\bigg|\bigg\|_{\psi_1},
\end{align*}
where $\mathbb{N}^0_{\epsilon_4}$ is an $\epsilon_4$-net of $\mathbb{V}(s,d)$ with properties in Lemma \ref{lem:discret}. Using Lemma \ref{lem:maximal_ineq}, we have
\begin{align*}
\bigg\| \max_{1\leq i\leq n}\sup_{\bv\in\mathbb{V}(s,d)}|g_{\bv}(\bX_i)| \bigg\|_{\psi_1}\lesssim \bigg(s\log\frac{ed}{\epsilon_4 s} + \log n \bigg) \sup_{\bv\in\mathbb{V}(s,d)}\bigg\|\frac{\bv^{\T}\bX_i\bX_i^{\T}\bv}{\bv^{\T}\bSigma\bv}-1\bigg\|_{\psi_1}.
\end{align*}
It follows that
\begin{align}\label{eq:lem34}
\bigg\|\max_{1\leq i\leq n}\sup_{\bv\in\mathbb{V}(s,d)}|g_{\bv}(\bX_i)| \bigg\|_{\psi_1}\lesssim K_2 \big\{ s\log( \gamma_s \cdot ed/s) + \log n\big\}.
\end{align}
Combining \eqref{eq:lem32}, \eqref{eq:lem33}, and \eqref{eq:lem34} gives
\begin{align}\label{eq:lem35}
\P\bigg[\hat{Q}_{\max} \leq2\E \hat{Q}_{\max} + K_2 \max\bigg\{2\sqrt{2 t}, C_{11} \gamma_n(s,d) \frac{t}{\sqrt{n}}\bigg\}\bigg]\geq1-4e^{-t},
\end{align}
where we recall that $\gamma_n(s,d)=s\log(\gamma_s \cdot ed/s)\vee s\log n$. 

Now we bound the expectation $\E \hat{Q}_{\max} $. Here we use a result that involves the generic chaining complexity, $\gamma_2(T,\rho)$, of a metric space $(T,\rho)$. See Definition 2.2.19 in \cite{talagrand2014upper}. We refer the readers to \cite{talagrand2014upper} for a systematic introduction. Note that 
\begin{align*}
\sup_{\bv\in\mathbb{V}(s,d)}\|\bv_{\bSigma}^{\T}\bX_i\|_{\psi_1}=\sup_{\bv\in\mathbb{V}(s,d)}\|(\bSigma^{1/2}\bv_{\bSigma})^{\T}\bU_i\|_{\psi_1}\leq K_1,
\end{align*}
and
\begin{align*}
\|(\bv_{\bSigma}-\tilde\bv_{\bSigma})^{\T}\bX_i\|_{\psi_2}=\|(\bv_{\bSigma}-\tilde\bv_{\bSigma})^{\T}\bSigma^{1/2}\bU_i\|_{\psi_2}\leq K_1  \|\bv_{\bSigma}-\tilde\bv_{\bSigma}\|_{\bSigma},
\end{align*}
for any $\bv,\tilde\bv\in\mathbb{V}(s,d)$. It follows from Lemma \ref{lem:gc1} and Lemma \ref{lem:gc2} that
\begin{align}\label{eq:lem36}
 \E \hat{Q}_{\max} &= \sqrt{n} \, \E\bigg\{ \!\!\sup_{\bv\in\mathbb{V}(s,d)}\bigg|\frac{1}{n}\sum_{i=1}^n\big(\bv_{\bSigma}^{\T}\bX_i\big)^2\!\!-\!\!1\bigg|\bigg\} \notag \\
&\lesssim K_1^2\bigg\{\! \gamma_2(\mathbb{V}(s,d),\|\cdot\|_{\bSigma})\!+\!
\frac{\gamma_2(\mathbb{V}(s,d),\|\cdot\|_{\bSigma})^2}{\sqrt{n}}\bigg\}.
\end{align}
By Lemma \ref{lem:gc3}, we have
\begin{align}\label{eq:lem37}
\gamma_2(\mathbb{V}(s,d),\|\cdot\|_{\bSigma})\lesssim\E\bigg\{ \sup_{\bv\in\mathbb{V}(s,d)}(\bv_{\bSigma}^{\T}\bZ) \bigg\},
\end{align}
where $\bZ\stackrel{{\sf d}}{=}N_d(\zero,\bSigma)$. Similar to the proof of Lemma \ref{lem:discret}, we have
\begin{align}\label{eq:lem38}
\E \bigg\{ \sup_{\bv\in\mathbb{V}(s,d)}(\bv_{\bSigma}^{\T}\bZ ) \bigg\} \leq\frac{4}{3}\E\bigg\{ \max_{\bv\in\mathbb{N}^0_{\epsilon_4}}(\bv_{\bSigma}^{\T}\bZ) \bigg\} \leq 2(  \log |\mathbb{N}^0_{\epsilon_4}|  )^{1/2}\lesssim  \gamma^{1/2}_n(s,d) ,
\end{align}
where $\epsilon_4= ( 4\gamma_s )^{-1}$. Together, \eqref{eq:lem36}, \eqref{eq:lem37}, and \eqref{eq:lem38} imply that
\begin{align}\label{eq:lem39}
\E \hat{Q}_{\max} \lesssim K_1^2  \, \gamma^{1/2}_n(s,d).
\end{align}
Combining \eqref{eq:lem35} and \eqref{eq:lem39}, we deduce that
\begin{align*}
\P\bigg[ \hat{Q}_{\max} \leq C_{12}K_2 \, \gamma^{1/2}_n(s,d)+ K_2 \max\bigg\{2\sqrt{2 t}, C_{11}  \gamma_n(s,d) \frac{t}{\sqrt{n}}\bigg\}\bigg]\geq1-4e^{-t}.
\end{align*}
This completes the proof.
\end{proof}

\subsubsection{Proof of Lemma \ref{lem:coupling}}
\begin{proof}
Recall that
\begin{align*}
M_{\max,\epsilon} = \max_{1\leq j\leq p_{\epsilon}}\bigg|\frac{1}{\sqrt{n}}\sum_{i=1}^n R_{ij}\bigg| , 
\end{align*}
and $\E R_{ij} =0$ for $i=1,\ldots,n$ and $j=1,\ldots,p_{\epsilon}$. Moreover, define $R_{ij} =-R_{i,j-p_{\epsilon}}$ for $j=p_{\epsilon}+1,\ldots,2p_{\epsilon}$ and put $\bR_i =(R_{i1},\ldots,R_{i,2p_{\epsilon}})^{\T}$ for $i=1,\ldots,n$. Let $\bG =(G_1,\ldots,G_{p_{\epsilon}},-G_1,\ldots,-G_{p_{\epsilon}})^{\T}$ be a $(2p_{\epsilon})$-dimensional Gaussian random vector satisfying
\begin{align*}
\E ( G_j G_k )=\E (R_{ij}R_{ik}), \ \  1\leq j\leq k\leq p_{\epsilon} .
\end{align*}
Applying Lemma \ref{lem:coupling_ineq} to $\{\bR_i\}_{i=1}^n$ and $\bG$, we have, for any $\delta>0$,
\begin{align}\label{eq:coupling1}
&\P\bigg(\bigg|\max_{1\leq j\leq p_{\epsilon}} \bigg| \frac{1}{\sqrt{n}}\sum_{i=1}^n R_{ij} \bigg|-\max_{1\leq j\leq p_{\epsilon}}|G_j|\bigg| \geq 16\delta\bigg)\notag\\
\lesssim& \, D_1 \frac{\log(2p_{\epsilon}\vee n)}{\delta^2n}+ (D_2+D_3) \frac{ \{ \log(2p_{\epsilon}\vee n)\}^2}{\delta^3n^{3/2}}+\frac{\log n}{n},
\end{align}
where we put
\begin{align*}
D_1&=\E\bigg[\max_{1\leq j, k\leq 2p_{\epsilon}}\bigg|\sum_{i=1}^n \{ R_{ij}R_{ik}-\E ( R_{ij}R_{i k}) \} \bigg| \bigg],\\
D_2&=\E\bigg( \max_{1\leq j\leq 2p_{\epsilon}}\sum_{i=1}^n|R_{ij}|^3\bigg) ,\\
D_3&=\sum_{i=1}^n \E\bigg[\max_{1\leq j\leq 2p_{\epsilon}}|R_{ij}|^3  \mathds{1}\bigg\{ \max_{1\leq j\leq 2p_{\epsilon}}|R_{ij}|>\frac{\delta n^{1/2}}{\log(2p_{\epsilon}\vee n)}\bigg\} \bigg].
\end{align*}
Note that, for $i=1,\ldots,n$,
\begin{align*}
\E\bigg( \max_{1\leq j\leq 2p_{\epsilon}}R_{ij}^4\bigg) \geq\frac{\delta n^{1/2}}{\log(2p_{\epsilon}\vee n)} \E\bigg[\max_{1\leq j\leq 2p_{\epsilon}}|R_{ij}|^3  \mathds{1}\bigg\{ \max_{1\leq j\leq 2p_{\epsilon}}|R_{ij}|>\frac{\delta n^{1/2}}{\log(2p_{\epsilon}\vee n)}\bigg\} \bigg],
\end{align*}
we have
\begin{align*}
D_3\leq\frac{\log(2p_{\epsilon}\vee n)}{\delta n^{1/2}} \sum_{i=1}^n\E\bigg( \max_{1\leq j\leq 2p_{\epsilon}}R_{ij}^4\bigg).
\end{align*}
Hence, we deduce from \eqref{eq:coupling1} that
\begin{align}\label{eq:coupling2}
\P\bigg(\bigg|\max_{1\leq j\leq p_{\epsilon}}\bigg|\frac{1}{\sqrt{n}}\sum_{i=1}^n R_{ij}\bigg|-\max_{1\leq j\leq p_{\epsilon}}|G_j|\bigg|\geq16\delta\bigg)\lesssim & ~ \, D_1 \frac{\log(2p_{\epsilon}\vee n)}{\delta^2n}+D_2 \frac{\{ \log(2p_{\epsilon}\vee n)\}^2}{\delta^3n^{3/2}}\notag\\
&~ + D_4  \frac{\{ \log(2p_{\epsilon}\vee n)\}^3}{\delta^4 n^2}+\frac{\log n}{n},
\end{align}
where
\begin{align*}
D_4:=\sum_{i=1}^n\E\bigg( \max_{1\leq j\leq 2p_{\epsilon}}R_{ij}^4\bigg) .
\end{align*}

Next we bound $D_1$, $D_2$, and $D_4$, starting with $D_1$. By Lemma \ref{lem:cher1},
\begin{align}\label{eq:coupling_B1}
D_1\lesssim&~ \sqrt{\log ( 2p_{\epsilon} )} \max_{1\leq j\leq 2p_{\epsilon}}\bigg(\sum_{i=1}^n\E R_{ij}^4 \bigg)^{1/2}+\log  ( 2p_{\epsilon} )  \bigg\{ \E \bigg( \max_{1\leq i\leq n}\max_{1\leq j\leq 2p_{\epsilon}}  R_{ij}^4 \bigg) \bigg\}^{1/2}\notag\\
=& ~ \sqrt{\log  ( 2p_{\epsilon} )} \underbrace{\max_{1\leq j\leq p_{\epsilon}}\bigg(\sum_{i=1}^n\E R_{ij}^4 \bigg)^{1/2}}_{D_{11}}+\log ( 2p_{\epsilon} )  \underbrace{\bigg\{ \E \bigg( \max_{1\leq i\leq n}\max_{1\leq j\leq p_{\epsilon}}R_{ij}^4 \bigg) \bigg\}^{1/2}}_{D_{12}}.
\end{align}
For $D_{11}$, using Lemma \ref{lem:orlicz norm}, we deduce that
\begin{align*}
\E  R_{ij}^4 \leq	4!  \|R_{ij}\|_{\psi_1}^4= 4!   \bigg\|\frac{\bv_j^{\T}\bSigma^{1/2}\bU_i\bU_i^{\T}\bSigma^{1/2}\bv_j}{\bv_j^{\T}\bSigma\bv_j}-1\bigg\|_{\psi_1}^4.
\end{align*}
This gives $\E  R_{ij}^4 \leq   4! \, \sup_{\bv\in\mathbb{S}^{d-1}}\|\bv^{\T}\bU_i\bU_i^{\T}\bv-1\|_{\psi_1}^4=4! \, K_2^4$ and hence
\begin{align}\label{eq:coupling_B11}
D_{11}\leq (4!\, K_2^4\, n)^{1/2}=24^{1/2}K_2^2 \, \sqrt{n}.
\end{align}
To bound $D_{12}$, by Lemmas \ref{lem:orlicz norm} and \ref{lem:maximal_ineq}, we have
\begin{align*}
\E \bigg( \max_{1\leq i\leq n}\max_{1\leq j\leq p_{\epsilon}}R_{ij}^4 \bigg)  \leq &~ 4! \, \bigg\|\max_{1\leq i\leq n}\max_{1\leq j\leq p_{\epsilon}}|R_{ij}|  \bigg\|_{\psi_1}^4 \notag \\
\leq & ~4! \,  K_0^4 \{ \log(np_{\epsilon}+1)\}^4 \max_{1\leq i\leq n}\max_{1\leq j\leq p_{\epsilon}}\|R_{ij}\|_{\psi_1}^4\notag\\
\leq& ~ 4! \,  K_0^4 \{ \log(np_{\epsilon}+1)\}^4 \sup_{\bv\in\mathbb{S}^{d-1}}\|\bv^{\T}\bU\bU^{\T}\bv-1\|_{\psi_1}^4\notag\\
=&~ 4! \, K_0^4K_2^4 \{ \log(np_{\epsilon}+1) \}^4 ,
\end{align*}
which further implies
\begin{align}\label{eq:coupling_B12}
D_{12}\leq24^{1/2}K_0^2K_2^2 \{ \log(np_{\epsilon}+1) \}^2.
\end{align}
Combining \eqref{eq:coupling_B1}, \eqref{eq:coupling_B11}, and \eqref{eq:coupling_B12} yields
\begin{align}\label{eq:coupling3}
D_1\lesssim K_2^2 (\log2p_{\epsilon})^{1/2} \sqrt{n} +  K_0^2K_2^2 \{ \log(np_{\epsilon}+1)\}^2	\log(2p_{\epsilon}).
\end{align}
For $D_2$, it follows from Lemma \ref{lem:cher2} that
\begin{align}\label{eq:coupling_B2}
D_2\!=\!\E\bigg(\! \max_{1\leq j\leq p_{\epsilon}}\sum_{i=1}^n|R_{ij}|^3\!\bigg) \!\lesssim\! \max_{1\leq j\leq p_{\epsilon}}\underbrace{\sum_{i=1}^n\E |R_{ij}|^3 }_{D_{21}}+ ( \log p_{\epsilon} ) \underbrace{\E\bigg(\! \max_{1\leq i\leq n}\max_{1\leq j\leq p_{\epsilon}}|R_{ij}|^3\!\bigg)}_{D_{22}}.
\end{align}
By Lemma \ref{lem:orlicz norm}, we have $\E |R_{ij}|^3 \leq 3! \, K_2^3$ and hence
\begin{align}\label{eq:coupling_B21}
D_{21}\leq 3! \, K_2^3 \, n.
\end{align}
Further, in view of Lemma \ref{lem:maximal_ineq}, we have
\begin{align}\label{eq:coupling_B22}
D_{22}\leq  3! \,  K_0^3 \{ \log(np_{\epsilon}+1) \}^3 \max_{1\leq i\leq n}\max_{1\leq j\leq p_{\epsilon}}\|R_{ij}\|_{\psi_1}^3 \leq   3! \, K_0^3K_2^3 \{ \log(np_{\epsilon}+1)\}^3.
\end{align}
Together, \eqref{eq:coupling_B2}, \eqref{eq:coupling_B21}, and \eqref{eq:coupling_B22} yield that
\begin{align}\label{eq:coupling4}
D_2\lesssim K_2^3 \, n + K_0^3 K_2^3 \{ \log(np_{\epsilon}+1)\}^3\log p_{\epsilon}.
\end{align}
For $D_4$, using Lemmas \ref{lem:orlicz norm} and \ref{lem:maximal_ineq}, we deduce that

\begin{align*}
& \E\bigg( \max_{1\leq j\leq 2p_{\epsilon}}R_{ij}^4\bigg)  =  \E\bigg( \max_{1\leq j\leq p_{\epsilon}}R_{ij}^4\bigg) \leq 4! \, \bigg\|\max_{1\leq j\leq p_{\epsilon}}|R_{ij}|  \bigg\|_{\psi_1}^4 \notag \\
 & \leq   4! \, K_0^4 \{ \log(p_{\epsilon}+1)\}^4 \sup_{\bv\in\mathbb{S}^{d-1}}\|\bv^{\T}\bU\bU^{\T}\bv-1\|_{\psi_1}^4 =   4! \,  K_0^4K_2^4 \{ \log(p_{\epsilon}+1)\}^4.
\end{align*}
Consequently, we have
\begin{align}\label{eq:coupling5}
D_4\lesssim K_0^4 K_2^4 \,  \{ \log(p_{\epsilon}+1)\}^4 n .
\end{align}

Finally, putting \eqref{eq:coupling2}, \eqref{eq:coupling3}, \eqref{eq:coupling4}, and \eqref{eq:coupling5} together, we obtain
\begin{align*}
&\P\bigg(\bigg|\max_{1\leq j\leq p_{\epsilon}}\bigg|\frac{1}{\sqrt{n}}\sum_{i=1}^n R_{ij}\bigg|-\max_{1\leq j\leq p_{\epsilon}}|G_j|\bigg|\geq16\delta\bigg)\notag\\
\lesssim & K_2^2 \frac{ \sqrt{ \log ( 2p_{\epsilon} ) } \log(2p_{\epsilon}\vee n)}{\delta^2 \sqrt{n}}+ K_0^2K_2^2 \frac{\log( 2p_{\epsilon})\{ \log(np_{\epsilon}+1)\}^2\log(2p_{\epsilon}\vee n)}{\delta^2n}  \\
& \quad  + K_2^3\frac{ \{ \log(2p_{\epsilon}\vee n)\}^2}{\delta^3 \sqrt{n}} + K_0^3K_2^3 \frac{ (\log p_{\epsilon}) \{ \log(np_{\epsilon}+1)\}^3 \{ \log(2p_{\epsilon}\vee n) \}^2}{\delta^3 n^{3/2}} \\
& \quad + K_0^4K_2^4 \frac{ \{ \log(p_{\epsilon}+1)\}^4 \{ \log(2p_{\epsilon}\vee n)\}^3}{\delta^4n}+\frac{\log n}{n},
\end{align*}
as desired.
\end{proof}

\subsubsection{Proof of Lemma \ref{lem:LT2}}

\begin{proof}
Define $\bar W_{ik} =W_{ik} \mathds{1}\big(|W_{ik}|\leq\tau \{ \log(p_1+n) \}^{1/2}\big)$ for some sufficiently large $\tau$. Then, for some constant $C_{22}>0$, 
\begin{align*}
&\P\left\{ \max_{1\leq j, k\leq p_1}\bigg|\frac{1}{n}\sum_{i=1}^n(W_{ij}W_{ik})^2-\E (W_{ij}W_{ik})^2 \bigg|\geq C_{22}\sqrt{\frac{\log p_1}{n}}\right\} \\
\leq & \, np_1 \max_{\substack{1\leq i\leq n, 1\leq k\leq p_1 }}\P\big[ |W_{ik}|\!>\!\tau \{ \log(p_1+n)\}^{1/2}\big]\\
&+\P\left\{ \max_{1\leq j, k\leq p_1}\bigg|\frac{1}{n}\sum_{i=1}^n(W_{ij}\bar W_{ik})^2-\E (W_{ij}W_{ik})^2 \bigg|\!\geq\! C_{22}\sqrt{\frac{\log p_1}{n}} \right\} .
\end{align*}
Using Cauchy-Schwarz inequality, we deduce that, for any $\eta>0$,
\begin{align*}
\big|\E (W_{ij}W_{ik})^2 -\E (W_{ij}\bar W_{i k})^2 \big|\leq& (\E W_{ij}^4 )^{1/2}\cdot\big[\E\big\{ W_{ik}^4\cdot\mathds{1}\big(|W_{ik}|>\tau\{ \log(p_1+n)\}^{1/2}  \big) \big\} \big]^{1/2}\\
\leq& (\E W_{ij}^4 )^{1/2} (n+p_1)^{- \tau^2\eta/4 }\cdot\big[\E\big\{ W_{ik}^4\exp( \eta W_{ik}^2 / 2 )\big\} \big]^{1/2}.
\end{align*}
By the elementary inequality $x^2e^x\leq e^{2x}$, $x>0$, we have, for any $\eta>0$,
\begin{align*}
\big|\E (W_{ij}W_{i k})^2 - \E (W_{ij}\bar W_{ik})^2  \big|	  \leq (\E W_{ij}^4 )^{1/2} (n+p_1)^{- \tau^2\eta/4 }\cdot 2\eta^{-1} \{ \E \exp(\eta W_{ik}^2) \}^{1/2}.
\end{align*}
Under Assumptions \ref{cdt1} and \ref{cdt2}, for any $\eta\in(0,K_1^{-2})$, there exists a  constant $C_{23} >0$ such that
\begin{align*}
\big|\E (W_{ij}W_{i k})^2 - \E (W_{ij}\bar W_{i k})^2 \big|\leq 2C_{23} \, \eta^{-1} (n+p_1)^{- \tau^2\eta/4 }.
\end{align*}
Hence, for all sufficiently large $\tau$, $n$, and $p_1$, we have
\begin{align*}
2C_{23} \, \eta^{-1} (n+p_1)^{- \tau^2\eta/4 } \leq \frac{C_{22}}{2}\sqrt{\frac{\log p_1}{n}}.
\end{align*}
It follows that
\begin{align*}
&\P \left\{ \max_{1\leq j, k\leq p_1}\bigg|\frac{1}{n}\sum_{i=1}^n(W_{ij}W_{ik})^2-\E (W_{ij}W_{ik})^2 \bigg|\geq C_{22}\sqrt{\frac{\log p_1}{n}}\right\} \\
\!\leq\!& \,  \underbrace{np_1 \max_{\substack{1\leq i\leq n, 1\leq k \leq p_1}}\P\big[ |W_{ik}|\!>\!\tau \{ \log(p_1\!+\!n) \}^{1/2}\big] }_{F_1}\\
&+\underbrace{\P\left\{ \max_{1\leq j,k\leq p_1}\bigg|\frac{1}{n}\sum_{i=1}^n(W_{ij}\bar W_{ik})^2\!-\!\E (W_{ij}\bar W_{ik})^2 \bigg|\!\geq\! \frac{C_{22}}{2}\sqrt{\frac{\log p_1}{n}}\right\} }_{F_2}.
\end{align*}
For $F_1$, we have, for any $\eta\in(0,K_1^{-2})$, $M>0$, and sufficiently large $\tau$,
\begin{align*}
F_1\leq np_1(n+p_1)^{-\tau^2\eta} \max_{\substack{1\leq i\leq n, 1\leq k \leq p_1}}\E\{ \exp(\eta W_{i k}^2) \} =O(p_1^{-M}).
\end{align*}
To bound $F_2$, it suffices to show that, for any $M>0$, there exists an absolute constant $C_{24}>0$ depending only on $M$ such that
\begin{align}\label{eq:LT1}
\P\left\{ \frac{1}{n}\sum_{i=1}^n(W_{ij}\bar W_{i k})^2-\E (W_{ij}\bar W_{i k})^2 \geq C_{24}\sqrt{\frac{\log p_1}{n}} \right\} =O(p_1^{-M-2}).
\end{align}
Define $W_{ijk} =(W_{ij}\bar W_{ik})^2-\E (W_{ij}\bar W_{ik})^2 $ for $1\leq i\leq n$ and $1\leq j, k\leq p_1$. By Markov's inequality, we have for any $t>0$,
\begin{align*}
& \P\left[ \frac{1}{n}\sum_{i=1}^n \{ (W_{ij}\bar W_{ik})^2-\E (W_{ij}\bar W_{ik})^2 \} \geq C_{25}\sqrt{\frac{\log p_1}{n}} \right] \\
	 \leq&	\exp\big( -C_{25} \, t\sqrt{n\log p_1} \, \big) \prod_{i=1}^n \E \exp(t W_{ijk}) .
\end{align*}
Using inequalities $e^x\leq1+x+x^2 e^{|x|}$ and $1+x\leq e^x$ for $x>0$, we deduce that
\begin{align*}
&\P\left\{ \frac{1}{n}\sum_{i=1}^n(W_{ij}\bar W_{ik})^2\!-\!\E (W_{ij}\bar W_{ik })^2 \geq C_{25}\sqrt{\frac{\log p_1}{n}} \right\} \\
\leq&\exp\big( -C_{25} \, t\sqrt{n\log p_1} \, \big)  \prod_{i=1}^n \big[ 1 + \E\{ t^2 W_{ijk}^2\exp(t|W_{ij k}|) \} \big] \\
\leq&\exp\bigg[ -C_{25} \, t\sqrt{n\log p_1}+\sum_{i=1}^n\E \{ t^2W_{ij k}^2\exp(t|W_{ijk}|) \} \bigg].
\end{align*}
Taking $t=\eta\tau^{-2}\sqrt{ ( \log p_1 ) /n}$ gives
\begin{align}\label{eq:LT2}
&P\left\{  \frac{1}{n}\sum_{i=1}^n(W_{ij}\bar W_{ik})^2-\E (W_{ij}\bar W_{ik})^2 \geq C_{25}\sqrt{\frac{\log p_1}{n}} \right\} \notag\\
\leq&\exp\bigg[  -C_{25}\frac{\eta}{\tau^2}\log p_1+\frac{\eta^2\log p_1}{\tau^4 n}\sum_{i=1}^n\underbrace{\E\bigg\{ W_{ij k}^2\exp\bigg(\frac{\eta}{\tau^2}\sqrt{\frac{\log p_1}{n}}|W_{ijk}|\bigg)\bigg\}}_{F_3} \bigg].
\end{align}
Using Cauchy-Schwarz inequality, we have
\begin{align*}
 F_3\leq\{\E(W_{ijk}^4)\}^{1/2}\bigg[\E\bigg\{ \exp\bigg(\frac{2\eta}{\tau^2}\sqrt{\frac{\log p_1}{n}}|W_{ij k}|\bigg)\bigg\}\bigg]^{1/2}  .
\end{align*}
According to Assumption \ref{cdt1}, for any $\eta\in(0,K_1^{-2})$ and sufficiently large $n$ and $p_1$ satisfying that $ \{ \log(p_1+n)\}^2 \log p_1 =o(n)$, there exists a constant $C_{26}>0$ depending on $K_1$, $\eta$, and $\tau$ such that
\begin{align}\label{eq:LT3}
 \E\bigg\{\exp\bigg(\frac{2\eta}{\tau^2}\sqrt{\frac{\log p_1}{n}}|W_{ij k}|\bigg)\bigg\}  \leq C_{26}.
\end{align}
Consequently, there exists a positive constant $C_{27}$ depending on $K_1$, $\eta$, and $\tau$ such that
\begin{align}\label{eq:LT4}
F_3\leq C_{27}.
\end{align}
Combining \eqref{eq:LT2} and \eqref{eq:LT4}, we obtain that for $\eta\in(0,K_1^{-2})$ and sufficiently large $\tau$, $n$, and $p_1$,
\begin{align*}
& P\left\{  \frac{1}{n}\sum_{i=1}^n(W_{ij}\bar W_{ik})^2-\E (W_{ij}\bar W_{ik})^2 \geq C_{25}\sqrt{\frac{\log p_1}{n}} \right\} \\
 \leq& \exp\bigg(-C_{25}\frac{\eta}{\tau^2}\log p_1+C_{27}\frac{\eta^2}{\tau^4}\log p_1\bigg)
\end{align*}
for any $C_{25}>0$. Therefore, for any $M>0$, there exists a constant $C_{24}>0$ depending only on $M$ such that \eqref{eq:LT1} holds. Similarly, it can be shown that
\begin{align*}
\P \left\{ \frac{1}{n}\sum_{i=1}^n(W_{ij}\bar W_{ik})^2-\E (W_{ij}\bar W_{ik})^2 \leq- C_{24}\sqrt{\frac{\log p_1}{n}} \right\} =O(p_1^{-M-2}).
\end{align*}
By taking $C_{22}=2C_{24}$, we get $F_2=O(p_1^{-M})$, which completes the proof.
\end{proof}

\subsubsection{Proof of Lemma \ref{lem:testinglem1}}
\begin{proof}
Similar to Lemma \ref{lem:discret}, we have, for any $\epsilon\in(0,1)$,
\begin{align}\label{eq:cqg3}
\cQ_{\max} \leq2\gamma_s(\bSigma_2)\epsilon\cdot\cQ_{\max} +\cM_{\max } .
\end{align}
Of note, we have
\begin{align*}
\cQ_{\max} \leq\sqrt{\frac{m}{n+m}} \underbrace{\sup_{\bv\in\mathbb{V}(s,d)}\bigg|\frac{\sqrt{n} \,\bv^{\T}(\hat\bSigma_1-\bSigma_2)\bv}{\bv^{\T}\bSigma_2\bv}\bigg|}_{\cQ_{\max, 1}}+\sqrt{\frac{n}{n+m}} \underbrace{\sup_{\bv\in\mathbb{V}(s,d)}\bigg|\frac{\sqrt{m} \,\bv^{\T}(\hat\bSigma_2-\bSigma_2)\bv}{\bv^{\T}\bSigma_2\bv}\bigg|}_{\cQ_{\max, 2}}.
\end{align*}
Using Lemma \ref{lem:connecting}, we deduce that, for any $t>0$,
\begin{align}\label{eq:n1bound}
\P\bigg(\cQ_{\max, 1}\leq C L_2 \bigg[    \gamma^{1/2}_{n}(s,d) +  \max\bigg\{  \sqrt{t},    \frac{t}{\sqrt{n}} \gamma_{n}(s,d) \bigg\}\bigg]\bigg)\geq1-4e^{-t},
\end{align}
and
\begin{align}\label{eq:n2bound}
\P\bigg(\cQ_{\max, 2}\leq C L_2 \bigg[    \gamma_{m}^{1/2}(s,d) + \max\bigg\{  \sqrt{t},      \frac{t}{\sqrt{m}}  \gamma_{m}(s,d) \bigg\}\bigg]\bigg)\geq1-4e^{-t},
\end{align}
where $C>0$ is an absolute constant, $L_2 =L_1^2+1$, $\gamma_{n}(s,d)=s\log \{ \gamma_s(\bSigma_2)\frac{ed}{s} \}\vee s\log n$, and $\gamma_{m}(s,d)=s\log\{ \gamma_s(\bSigma_2) \frac{ed}{s} \} \vee s\log m$. It follows that, for any $t_1,t_2>0$,
\begin{align*}
\P\bigg(\cQ_{\max}   \leq C L_2 \bigg[ &\, \gamma^{1/2}_{n}(s,d)  + \gamma_{m}^{1/2}(s,d) +\max\bigg\{ \sqrt{t_1},    \frac{t_1}{\sqrt{n}}  \gamma_{n}(s,d) \bigg\}\notag\\
 &\,  +\max\bigg\{ \sqrt{t_2},    \frac{t_2}{\sqrt{m}} \gamma_{m}(s,d) \bigg\}\bigg]  \bigg)\geq1-4e^{-t_1}-4e^{-t_2}.
\end{align*}
Taking $t_1=\log n$ and $t_2=\log m$ gives
\begin{align*}
\P\bigg[\cQ_{\max} \leq C_{31} L_2 \bigg\{  \gamma_{m}^{1/2}(s,d) + \sqrt{\log m}+  \frac{\log m}{\sqrt{m}} \gamma_{m}(s,d) \bigg\}  \bigg]	\geq1-\frac{4}{n}-\frac{4}{m},
\end{align*}
which proves \eqref{eq:cqg1}. Combining \eqref{eq:cqg1} and \eqref{eq:cqg3}, and taking $\epsilon=\epsilon_3= \{ m\gamma_s(\bSigma_2) \}^{-1}$, we obtain
\begin{align}\label{eq:cqg5}
\P\bigg( |\cQ_{\max} - \cM_{\max} | \leq C L_2\bigg[   \frac{ \gamma^{1/2}_{m}(s,d) }{m}+\max\bigg\{ \frac{\sqrt{ \log m}}{m},   & \, \frac{\log m}{m^{3/2}}\gamma_{m}(s,d) \bigg\}\bigg] \bigg)\notag\\
&\geq1-\frac{4}{n}-\frac{4}{m}.
\end{align}
Recalling the definition of $\cR_{ij}$ in \eqref{eq:Rijdef}, we have
\begin{align*}
\cM_{\max} =&\sqrt{\frac{nm}{n+m}}\max_{1\leq j\leq p_3}\bigg|\frac{\bu_j^{\T}(\hat\bSigma_1-\hat\bSigma_2)\bu_j}{\bu_j^{\T}\bSigma_2\bu_j}\bigg|=\sqrt{\frac{n}{m(n+m)}}\max_{1\leq j\leq p_3} \bigg| \sum_{i=1}^{m}\cR_{ij} \bigg|,
\end{align*}
where $\cM_{\max}$ is as in the proof of Theorem~\ref{thm:testing}. Moreover, there exists a $p_3$-dimensional Gaussian random vector $(\cG_1,\ldots,\cG_{p_3})^{\T}$ satisfying
\begin{align*}
\E( \cG_j\cG_k )=\frac{n}{m(n+m)} \E\bigg( \sum_{i=1}^{m}\cR_{ij}\cR_{ik}\bigg), \ \ 1\leq j\leq k \leq p_3,
\end{align*}
such that for every $\delta>0$,
\begin{align*}
&\P\bigg(\bigg| \cM_{\max}  - \max_{1\leq j\leq p_3}|\cG_j|\bigg|\geq16\delta\bigg)\notag\\
\lesssim & \, \bigg(\frac{n+m}{n}\bigg)^3 L_3^3 \frac{\{ \log(2p_3\vee m)\}^2}{\delta^3\{  m(n+m)/ n \}^{1/2}}+\bigg(\frac{n+m}{n}\bigg)^4 K_0^4  L_3^4 \frac{ \{ \log(p_3+1)\}^4 \{ \log(2p_3\vee m)\}^3}{\delta^4 \{  m(n+m) / n \}}\notag\\
\lesssim & \, L_2^3 \frac{ \{ \log(2p_3\vee m)\}^2}{\delta^3m^{1/2}} + K_0^4L_2^4\frac{ \{ \log(p_3+1)\}^4 \{ \log(2p_3\vee m) \}^3}{\delta^4m},
\end{align*}
where $L_3:=(m/n+1)L_2$. It follows that
\begin{align}\label{eq:cqg6}
\P\bigg(\bigg|  \cM_{\max}  -   \max_{1\leq j\leq p_3}|\cG_j| \bigg|\geq16\delta\bigg)\lesssim  L_2^3 \frac{ \gamma^2_{m}(s,d)}{\delta^3 m^{1/2}}+ K_0^4L_2^4 \frac{ \gamma^7_{m}(s,d)}{\delta^4m}.
\end{align}
Taking $\delta=\frac{K_0^4 \, \gamma^{5/8}_{m}(s,d)}{L_2 \, m^{1/8}}$, it follows from \eqref{eq:cqg5} and \eqref{eq:cqg6} that
\begin{align*}
\P\left[ \bigg|\cQ_{\max} - \max_{1\leq j\leq p_3}|\cG_j| \bigg|\geq C_{32} K_0^4 \frac{ \gamma^{5/8}_{m}(s,d)}{L_2 \, m^{1/8}} \right] \lesssim L_2^2 \frac{ \gamma^{1/8}_{m}(s,d)}{m^{1/8}}+L_2^2 \frac{ \gamma^{9/2}_{m}(s,d)}{m^{1/2}}.
\end{align*}
This completes the proof.
\end{proof}

\subsubsection{Proof of Lemma \ref{lem:testinglem2}}
\begin{proof}
Write $\hat{\bSigma} =  (m\hat\bSigma_1+n\hat\bSigma_2) / (n + m)$. By definition, we have
\begin{align}\label{eq:cqq2}
&|\hat\cQ_{\max} \!-\! \cQ_{\max} |\notag \\
=& \, \bigg|\sup_{\bv\in\mathbb{V}(s,d)}\bigg|\frac{\sqrt{\frac{nm}{n+m}} \bv^{\T}(\hat\bSigma_1-\hat\bSigma_2)\bv}{\frac{nm}{n+m} \bv^{\T}(\frac{\hat\bSigma_1}{n}+\frac{\hat\bSigma_2}{m})\bv}\bigg|-\sup_{\bv\in\mathbb{V}(s,d)}\bigg|\frac{\sqrt{\frac{nm}{n+m}} \bv^{\T}(\hat\bSigma_1-\hat\bSigma_2)\bv}{\bv^{\T}\bSigma_2\bv}\bigg|\bigg|\notag\\
\leq& \, \sqrt{\frac{nm}{n+m}} \sup_{\bv\in\mathbb{V}(s,d)}\bigg|\frac{\bv^{\T}(\hat\bSigma_1-\hat\bSigma_2)\bv}{\bv^{\T} \hat{\bSigma} \bv}-\frac{\bv^{\T}(\hat\bSigma_1-\hat\bSigma_2)\bv}{\bv^{\T}\bSigma_2\bv}\bigg|\notag\\
\leq& \, \sqrt{\frac{nm}{n+m}} \sup_{\bv\in\mathbb{V}(s,d)}\bigg|\frac{\bv^{\T}(\hat\bSigma_1-\hat\bSigma_2)\bv}{\bv^{\T}\bSigma_2\bv}\bigg| \sup_{\bv\in\mathbb{V}(s,d)}\bigg|\frac{\bv^{\T}\bSigma_2\bv}{\bv^{\T}\hat{\bSigma} \bv}\bigg| \times\notag\\
&  \bigg\{ \frac{m}{n+m} \sup_{\bv\in\mathbb{V}(s,d)}\bigg|\frac{\bv^{\T}(\hat\bSigma_1-\bSigma_2)\bv}{\bv^{\T}\bSigma_2\bv}\bigg|+\frac{n}{n+m} \sup_{\bv\in\mathbb{V}(s,d)}\bigg|\frac{\bv^{\T}(\hat\bSigma_2-\bSigma_2)\bv}{\bv^{\T}\bSigma_2\bv}\bigg|\bigg\}.
\end{align}
Combining \eqref{eq:n1bound} and \eqref{eq:n2bound}, and taking $t_1=\log n$ and $t_2=\log m$, we have
\begin{align*}
\P\left[ \sup_{\bv\in\mathbb{V}(s,d)}\bigg|\frac{\bv^{\T}  \hat{\bSigma} \bv}{\bv^{\T}\bSigma_2\bv}-1\bigg|\leq C_{34} L_2 \bigg\{   \frac{\gamma^{1/2}_{m}(s,d)}{\sqrt{m}}+ \frac{\log m}{m} \gamma_{m}(s,d) \bigg\} \right] \geq 1-\frac{4}{n}-\frac{4}{m},
\end{align*}
where $C_{34} > 0$ is an absolute constant, $\gamma_{m}(s,d)=s\log\{ \gamma_s(\bSigma_2) \frac{ed}{s} \}\vee s\log m$, and $L_2=L_1^2+1$. It follows that, for all sufficiently large $m$,
\begin{align} 
\P\left( \sup_{\bv\in\mathbb{V}(s,d)}\bigg|\frac{\bv^{\T}\bSigma_2\bv}{\bv^{\T}\hat{\bSigma}\bv}\bigg|\leq\bigg[ 1-C_{34} L_2 \bigg\{  \frac{ \gamma_{m}^{1/2}(s,d)}{\sqrt{m}}+ \frac{\log m}{m} \gamma_{m}(s,d) \bigg\} \bigg]^{-1}\right) 
\geq1-\frac{4}{n}-\frac{4}{m}. \notag
\end{align}
This, together with \eqref{eq:cqg1}, \eqref{eq:n1bound}, \eqref{eq:n2bound}, and \eqref{eq:cqq2}, proves \eqref{eq:cqq1}.
\end{proof}

\subsubsection{Proof of Lemma \ref{lem:testinglem3}}
\begin{proof}
Define $\cB_{\max}  = \sup_{\bv\in\mathbb{V}(s,d)} | \cB_{\bv}|$, where
\begin{align*}
\cB_{ \bv} =\sqrt{\frac{nm}{n+m}} \, \frac{\bv^\T\{  n^{-1} \sum_{i=1}^{n}\xi_i(\bX_i\bX_i^\T-\bSigma_2)- m^{-1} \sum_{i=1}^{m} \eta_i(\bY_i\bY_i^\T-\bSigma_2)\}\bv}{\bv^{\T}\bSigma_2\bv} 
\end{align*}
and $\xi_1,\ldots,\xi_{n}, \eta_1 ,\ldots, \eta_m$ are independent standard Gaussian random variables that are independent of $\{\bX_i\}_{i=1}^{n}$ and $\{\bY_i\}_{i=1}^{m}$. As in Lemma \ref{lem:testinglem1}, we have for $1\leq j\leq k \leq p_3$,
\begin{align*}
\E ( \cG_j\cG_k ) = & \, \frac{n}{m(n+m)} \E\bigg( \sum_{i=1}^{n}\cR_{ij}\cR_{ik}+\sum_{i=n+1}^{m}\cR_{ij}\cR_{ik}\bigg)\notag\\
=& \, \frac{m}{n+m} \E\bigg\{ \frac{(\bu_j^{\T}\bX_i)^2(\bu_k^{\T}\bX_i)^2}{\|\bu_j\|_{\bSigma_2}^2\|\bu_k \|_{\bSigma_2}^2}-1\bigg\} +   \frac{n^2}{m(n+m)} \E\bigg\{ \frac{(\bu_j^{\T}\bY_i)^2(\bu_k^{\T}\bY_i)^2}{\|\bu_j\|_{\bSigma_2}^2\|\bu_k \|_{\bSigma_2}^2}-1\bigg\} \notag\\
&+\frac{n(m-n)}{m(n+m)} \E\bigg\{ \frac{(\bu_j^{\T}\bY_i)^2(\bu_k^{\T}\bY_i)^2}{\|\bu_j\|_{\bSigma_2}^2\|\bu_k \|_{\bSigma_2}^2}-1\bigg\} \notag\\
=& \, \frac{m}{n+m} \E\bigg\{ \frac{(\bu_j^{\T}\bX_i)^2(\bu_k^{\T}\bX_i)^2}{\|\bu_j\|_{\bSigma_2}^2\|\bu_k\|_{\bSigma_2}^2}-1\bigg\} +\frac{n}{n+m} \E\bigg\{ \frac{(\bu_j^{\T}\bY_i)^2(\bu_k^{\T}\bY_i)^2}{\|\bu_j\|_{\bSigma_2}^2\|\bu_k\|_{\bSigma_2}^2}-1\bigg\} .
\end{align*}
Putting $\{ \cB_j \}_{j=1}^{p_3} = \{ \cB_{\bv} \}_{\bv \in \mathbb{N}^0_{\epsilon_3}}$, we have, for $j=1,\ldots,p_3$,
\begin{align*}
\cB_{j}=\sqrt{\frac{nm}{n+m}} \bigg[\frac{1}{n}\sum_{i=1}^{n}\xi_i \bigg\{ \frac{(\bu_j^{\T}\bX_i)^2}{\|\bu_j\|_{\bSigma_2}^2}-1\bigg\} -\frac{1}{m}\sum_{i=1}^{m}\eta_i \bigg\{ \frac{(\bu_j^{\T}\bY_i)^2}{\|\bu_j\|_{\bSigma_2}^2}-1\bigg\} \bigg].
\end{align*}
It follows that
\begin{align*}
&\E (  \cB_{ j}\cB_k | \, \bX_1,\ldots,\bX_{n},\bY_1,\ldots,\bY_{m} ) \\
=&\frac{m}{n+m} \frac{1}{n}\sum_{i=1}^{n}\bigg\{ \frac{(\bu_j^{\T}\bX_i)^2}{\|\bu_j\|_{\bSigma_2}^2}\!-\!1\bigg\}\bigg\{ \frac{(\bu_k^{\T}\bX_i)^2}{\|\bu_k\|_{\bSigma_2}^2}\!-\!1\bigg\}  \\
& + \frac{n}{n+m} \frac{1}{m}\sum_{i=1}^{m}\bigg\{ \frac{(\bu_j^{\T}\bY_i)^2}{\|\bu_j\|_{\bSigma_2}^2}\!-\!1\bigg\}\bigg\{ \frac{(\bu_k^{\T}\bY_i)^2}{\|\bu_k\|_{\bSigma_2}^2}\!-\!1\bigg\}.
\end{align*}
Define
\begin{align*}
\Delta_{\cG} =\max_{1\leq j\leq k\leq p_3} | \E ( \cG_j\cG_k )  -\E ( \cB_{ j}\cB_{k} | \, \bX_1,\ldots,\bX_{n},\bY_1,\ldots,\bY_{m}) |.
\end{align*}
Similar to the proof of Theorem \ref{thm:multiplier_bootstrap}, it can be shown that with probability greater than $1-O(p_3^{-M})$,
\begin{align*}
 \Delta_{\cG}^{1/3} \{ \log(2p_3)\}^{1/3} \{ \log(1/\Delta_{\cG})\}^{1/3}\leq\Delta_{p_3,m}, 
\end{align*}
where $\Delta_{p_3,m} \to 0$ as $p_3, m \to \infty$. By Lemma \ref{lem:comparison}, we have
\begin{align*}
\sup_{t\geq0}\bigg|\P\bigg(  \max_{1\leq j\leq p_3 }|\cG_{j}|  \leq t \bigg)-\P\bigg( \max_{1\leq j\leq p_3 } |\cB_{j}|\leq t \, \bigg|  \bX_1,\ldots,\bX_{n},\bY_1,\ldots,\bY_{m} \bigg)\bigg|=o_{\mathbb{P}}(1).
\end{align*}
Then, using Lemma \ref{lem:discret} and Lemma \ref{lem:testinglem2}, we deduce that
\begin{align*}
\sup_{t\geq0}\bigg|\P\bigg(\max_{1\leq j\leq p_3}|\cG_j| \leq t \bigg)-\P(\hat\cB_{\max}\leq t \, |  \bX_1,\ldots,\bX_{n},\bY_1,\ldots,\bY_{m})\bigg|=o_{\mathbb{P}}(1),
\end{align*}
as desired.
\end{proof}

\subsection{Auxillary lemmas}

In the sequel, we define $\reals^+$ and $\mathbb{Z}^+$ to be the sets of positive real values and integers. The following two lemmas are elementary, yet very useful, in the proofs of the above results.

\begin{lemma}\label{lem:sigma norm}
For any $\bv\in\mathbb{S}^{d-1}$, we have
\begin{align*}
\|\bv_{\bSigma}\|_{\bSigma}=1~~~{\rm and}~~~\bigg(\frac{\bv}{\|\bv\|_2}\bigg)_{\bSigma}=\bv_{\bSigma}.
\end{align*}
\end{lemma}
\begin{proof}
By definition, it is straightforward that $\|\bv_{\bSigma}\|_{\bSigma}=1$, and
\begin{align*}
\bigg(\frac{\bv}{\|\bv\|_2} \bigg)_{\bSigma}=\frac{\bv}{\bv^{\T}\bv}\bigg/\bigg\{ \frac{\bv^{\T}\bSigma\bv}{(\bv^{\T}\bv)^2} \bigg\}^{1/2}=\frac{\bv}{(\bv^{\T}\bSigma\bv)^{1/2}}=\bv_{\bSigma},
\end{align*}
as desired.
\end{proof}

\begin{lemma}\label{lem:orlicz norm}
For $\alpha\in [1,\infty)$, define the function $\psi_{\alpha}(x)=\exp(x^{\alpha})-1$, $x>0$. The Orlicz norm for a random variable $X$ is given by
\begin{align*}
\|X\|_{\psi_{\alpha}}:=\inf\bigg\{C>0:\E \bigg\{ \psi_{\alpha}\bigg(\frac{|X|}{C}\bigg) \bigg\} \leq1\bigg\}.
\end{align*}
Also, define the $L^p$ ($p \geq 1$) norm of a random variable $X$ by $ \|X\|_p = (  \E |X|^p   )^{1/p}$. Then, for every $p\in\mathbb{Z^+}$ we have
\begin{align*}
\|X\|_p\leq(p!)^{1/p}\cdot \|X\|_{\psi_1}.
\end{align*}
\end{lemma}
\begin{proof}
Note that for every $p\in\mathbb{Z}^+$ and $x\in\reals^+$, $x^p\leq p!\cdot\psi_1(x)$. Then, we have for any $C>0$,
\begin{align*}
\E \bigg( \frac{|X|^p}{C^p} \bigg) \leq p!\cdot\E\bigg\{ \psi_1\bigg(\frac{|X|}{C}\bigg) \bigg\} .
\end{align*}
The conclusion thus follows immediately.
\end{proof}

The following lemma is from \cite{vershynin2010introduction}.

\begin{lemma}\label{lem:covering numbers}
Let $(\Omega,\rho)$ be a metric space. For every $\epsilon>0$, a subset $\mathbb{N}_{\epsilon}(\Omega)$ of $\Omega$ is called an $\epsilon$-net of $\Omega$ if for every $\omega\in\Omega$, there is some $\xi\in \mathbb{N}_{\epsilon}(\Omega)$ such that $\rho(\omega,\xi)\leq\epsilon$. The minimal cardinality of an $\epsilon$-net $\Omega$, if finite, is called the covering number of $\Omega$ at scale $\epsilon$, and is denoted by $N(\Omega,\rho,\epsilon)$. The unit sphere $\mathbb{S}^{d-1}$ equipped with the Euclidean metric satisfies that for every $0<\epsilon\leq 1$, $N(\mathbb{S}^{d-1},\rho,\epsilon)\leq(1+2/\epsilon)^d$.
\end{lemma}

The following anti-concentration lemma is Theorem 3 in \cite{chernozhukov2014comparison} and is used in the proofs of Theorems \ref{thm:limiting_dist} and \ref{thm:testing}.
\begin{lemma}\label{lem:anticoncentration}
Let $(X_1,\ldots,X_d)^{\T}$ be a centered Gaussian random vector in $\reals^d$ with $\sigma_j^2:=\E ( X_j^2 ) >0$ for all $1\leq j\leq d$. Define $\underline{\sigma}   =\min_{1\leq j\leq d}\sigma_j$, $\overline{\sigma} =\max_{1\leq j\leq d}\sigma_j$, and $a_d =\E\{ \max_{1\leq j\leq d} (X_j/\!\sigma_j) \} $.\\
(i) If $\underline{\sigma}=\overline{\sigma}=\sigma$, then for every $\epsilon>0$,
\begin{align*}
\sup_{x\in \reals}\P\bigg( \bigg|\max_{1\leq j\leq d}X_j-x \bigg|\leq\epsilon \bigg)\leq \frac{4\epsilon}{\sigma}(a_d+1).
\end{align*}
(ii) If $\underline{\sigma}<\overline{\sigma}$, then for every $\epsilon>0$,
\begin{align*}
\sup_{x\in \reals}\P\bigg( \bigg|\max_{1\leq j\leq d}X_j-x \bigg|\leq\epsilon \bigg)\leq C\epsilon \big\{ a_d+\sqrt{1\vee\log(\underline{\sigma}/\epsilon)} \big\} ,
\end{align*}
where $C>0$ is a constant depending only on $\underline{\sigma}$ and $\overline{\sigma}$.
\end{lemma}

The following lemma from \cite{law2007tail} is used in the proof of Lemma \ref{lem:connecting}.
\begin{lemma}\label{lem:tail_ineq}
Let $\bX_1,\ldots,\bX_n$ be independent random variables taking values in a measurable space $(\mathbb{S},\cB)$, and let $\cF$ be a countable class of measurable functions $f:\mathbb{S}\rightarrow \reals$. Assume that for $i=1,\ldots, n$, $\E f(\bX_i)=0$ for every $f\in\cF$ and $\|\sup_{f\in\cF}|f(\bX_i)|\|_{\psi_1}<\infty$. Define
\begin{align*}
Z =\sup_{f\in\cF} \bigg| \sum_{i=1}^n f(\bX_i) \bigg|~~{\rm and}~~\sigma^2 =\sup_{f\in\cF}\sum_{i=1}^n \E f^2(\bX_i).
\end{align*}
Then, for every $0<\eta<1$ and $\delta>0$, there exists a constant $C=C(\eta,\delta)$ such that for all $t\geq 0$,
\begin{align*}
& \P\{ Z\geq(1+\eta)\E Z+t \}  \\
 \leq& \exp\bigg\{ -\frac{t^2}{2(1+\delta)\sigma^2}\bigg\} + 3\exp\bigg\{ -\frac{t}{C\|\max_{1\leq i\leq n}\sup_{f\in\cF}|f(\bX_i)|\|_{\psi_1}}\bigg\} ,
\end{align*}
and
\begin{align*}
& \P\{ Z\leq(1-\eta)\E Z-t \} \\
  \leq& \exp\bigg\{ -\frac{t^2}{2(1+\delta)\sigma^2}\bigg\} +3\exp\bigg\{ -\frac{t}{C\|\max_{1\leq i\leq n}\sup_{f\in\cF}|f(\bX_i)|\|_{\psi_1}}\bigg\}.
\end{align*}
\end{lemma}

The following lemma is Lemma 2.2.2 in \cite{van1996weak} and is used in the proofs of Lemma \ref{lem:connecting} and Lemma \ref{lem:coupling}.
\begin{lemma}\label{lem:maximal_ineq}
For any $\alpha\in[1,\infty)$, there exists a constant $K_0>0$ depending only on $\alpha$ such that
\begin{align*}
\bigg\|\max_{1\leq i\leq n}X_i \bigg\|_{\psi_{\alpha}}\leq K_0  \, \psi_{\alpha}^{-1}(n)\max_{1\leq i\leq n}\|X_i\|_{\psi_{\alpha}}.
\end{align*}
\end{lemma}

The following lemma is Theorem A in \cite{mendelson2010empirical} and is used in the proof of Lemma \ref{lem:connecting}.
\begin{lemma}\label{lem:gc1}
Let $\cF$ be a class of mean-zero functions on a probability space $(\reals^d, \mu, \P)$, and let $\bX_1,\ldots,\bX_n$ be independent random variables in $\reals^d$ distributed according to $\P$. Then, there exists an absolute constant $C>0$ such that
\begin{align*}
\E\bigg\{ \sup_{f\in\cF} \bigg|\frac{1}{n}\sum_{i=1}^n f^2(\bX_i)-\E f^2 \bigg| \bigg\} \leq C \bigg\{  \sup_{f\in\cF}\|f\|_{\psi_1} \frac{\gamma_2(\cF,\psi_2)}{\sqrt{n}}+\frac{\gamma_2^2(\cF,\psi_2)}{n} \bigg\} .
\end{align*}
The complexity parameter $\gamma_2(\cF, \psi_2)$ of $\cF$ is the $\gamma_2$ functional with respect to the $\psi_2$ norm. See \cite{talagrand2014upper} for its definition and properties.
\end{lemma}

The following two lemmas are Theorem 2.7.5 and Theorem 2.4.1 in \cite{talagrand2014upper} on generic chaining, and are used in the proof of Lemma \ref{lem:connecting}.
\begin{lemma}\label{lem:gc2}
If $f:(T,\rho)\mapsto(U, \varrho)$ is surjective and there exists a constant $C>0$ such that
\begin{align*}
\varrho(f(x),f(y))\leq C \rho(x,y),
\end{align*}
for any $x,y\in T$. Then, we have
\begin{align*}
\gamma_{\alpha}(U, \varrho)\leq C  K(\alpha) \, \gamma_{\alpha}(T,\rho),
\end{align*}
where $K(\alpha)$ is an absolute constant depending only on $\alpha$.
\end{lemma}

\begin{lemma}\label{lem:gc3}
For any metric space $(T,\rho)$ and centered Gaussian process $\{X_t\}_{t\in T}$, there exist universal constants $C>c>0$ such that
\begin{align*}
c \, \gamma_2(T,\rho)\leq \E\bigg( \sup_{t\in T} X_t \bigg) \leq C\,\gamma_2(T,\rho).
\end{align*}
\end{lemma}

The following two lemmas are Lemma 1 and Lemma 9 in \cite{chernozhukov2014gaussian} and are used in the proof of Lemma \ref{lem:coupling}.
\begin{lemma}\label{lem:cher1}
Let $\bX_1,\ldots,\bX_n$ be independent centered random vectors in $\reals^d$ with $d\geq2$. Then, there exists a absolute constant $C>0$ such that
\begin{align*}
&\E\bigg[\max_{1\leq j, k\leq d}\bigg|\frac{1}{n}\sum_{i=1}^n \{ X_{ij}X_{ik}\!-\!\E( X_{ij}X_{ik}) \} \bigg| \bigg]\\
\leq& C\bigg[ \sqrt{\frac{\log d}{n}} \max_{1\leq j\leq d} \bigg\{ \frac{1}{n}\sum_{i=1}^n\E(X_{ij}^4)\bigg\}^{1/2}\!+\!\frac{\log d}{n}  \bigg\{ \E \bigg( \max_{1\leq i\leq n}\max_{1\leq j\leq d}X_{ij}^4 \bigg) \bigg\}^{1/2}\bigg].
\end{align*}
\end{lemma}

\begin{lemma}\label{lem:cher2}
Let $\bX_1,\ldots,\bX_n$ be independent random vectors in $\reals^d$ with $d\geq2$ such that $X_{ij}\geq0$ for all $1\leq i\leq n$ and $1\leq j\leq d$. Then
\begin{align*}
\E\bigg( \max_{1\leq j\leq d}\sum_{i=1}^n X_{ij} \bigg) \lesssim\max_{1\leq j\leq d} \sum_{i=1}^n \E( X_{ij} ) + (\log d) \cdot \E\bigg( \max_{1\leq i\leq n}\max_{1\leq j\leq d}X_{ij} \bigg).
\end{align*}
\end{lemma}

The following lemma is Theorem 2 in  \cite{chernozhukov2014comparison} and is used in the proofs of Theorem \ref{thm:multiplier_bootstrap} and Lemma \ref{lem:testinglem3}.
\begin{lemma}\label{lem:comparison}
Let $\bX=(X_1,\ldots,X_d)^{\T}$ and $\bY=(Y_1,\ldots,Y_d)^{\T}$ be centered Gaussian random vectors in $\reals^d$ with covariance matrices $\bSigma^{\bX}=(\sigma_{jk}^{\bX})_{1\leq j,k\leq d}$ and $\bSigma^{\bY}=(\sigma_{jk}^{\bY})_{1\leq j,k\leq d}$, respectively. Suppose that $d\geq2$ and $\sigma_{jj}^{\bY}>0$ for all $1\leq j\leq d$. Define
\begin{align*}
a_d =\E\bigg\{ \max_{1\leq j\leq d}(Y_j/\sigma_{jj}^{\bY}) \bigg\} ~ \mbox{ and }~ \Delta =\max_{1\leq j,k\leq d}|\sigma_{jk}^{\bX}-\sigma_{jk}^{\bY}|.
\end{align*}
Then
\begin{align*}
\sup_{x\in\reals}\bigg|\P\bigg(\max_{1\leq j\leq d}X_j\leq x\bigg)-\P\bigg(\max_{1\leq j\leq d}Y_j\leq x \bigg) \bigg|\leq C\Delta^{1/3}(\log d)^{1/3} \{1\vee a_d^2\vee\log(1/\Delta) \}^{1/3},
\end{align*}
where $C>0$ is an absolute constant depending only on $\min_{1\leq j\leq d}\sigma_{jj}^{\bY}$ and $\max_{1\leq j\leq d}\sigma_{jj}^{\bY}$. In particular, we have $a_d\leq(2\log d)^{1/2}$ and
\begin{align*}
\sup_{x\in\reals}\bigg|\P\bigg(\max_{1\leq j\leq d}X_j\leq x \bigg)-\P\bigg(\max_{1\leq j\leq d}Y_j\leq x \bigg)\bigg|\leq C'\Delta^{1/3} \{ 1\vee\log(d/\Delta) \}^{2/3},
\end{align*}
where $C'>0$ is an absolute constant depending only on $\min_{1\leq j\leq d}\sigma_{jj}^{\bY}$ and $\max_{1\leq j\leq d}\sigma_{jj}^{\bY}$.
\end{lemma}

\section*{Acknowledgements}
The authors sincerely thank the Editor, Associate Editor, and an anonymous referee for their valuable comments and suggestions. A part of this work was carried out when Fang Han was visiting Department of Biostatistics at Johns Hopkins University.

\end{document}